\documentclass[bj]{imsart}

\RequirePackage{amsthm,amsmath,amsfonts,amssymb}
\RequirePackage[authoryear]{natbib} %% uncomment this for author-year bibliography
\RequirePackage{graphicx}
\usepackage{subfig}
\startlocaldefs
\theoremstyle{plain}
\newtheorem{theorem}{Theorem}

\newtheorem{condition}[theorem]{Condition}

\newtheorem{lemma}[theorem]{Lemma}

\newtheorem{proposition}[theorem]{Proposition}
\newtheorem{remark}[theorem]{Remark}

\numberwithin{equation}{section}
\numberwithin{theorem}{section}

\newcommand{\E}{\operatorname{E}}

\renewcommand{\mathbf}{\boldsymbol}

\newcommand{\abs}[1]{\left\vert#1\right\vert}

\newcommand{\suit}[1]{\left(#1\right)}

\newcommand{\td}{\stackrel{d}{\to}}

\endlocaldefs

\begin{document}

\begin{frontmatter}
\title{All block maxima method for estimating the extreme value index}
%\title{A sample article title with some additional note\thanksref{T1}}
\runtitle{All block maxima method for estimating the extreme value index}
%\thankstext{T1}{A sample of additional note to the title.}

\begin{aug}
%%%%%%%%%%%%%%%%%%%%%%%%%%%%%%%%%%%%%%%%%%%%%%
%%Only one address is permitted per author. %%
%%Only division, organization and e-mail is %%
%%included in the address.                  %%
%%Additional information can be included in %%
%%the Acknowledgments section if necessary. %%
%%%%%%%%%%%%%%%%%%%%%%%%%%%%%%%%%%%%%%%%%%%%%%
\author[A]{\fnms{Jochem} \snm{Oorschot}\ead[label=e1, mark]{oorschot@ese.eur.nl}},
\author[B]{\fnms{Chen} \snm{Zhou}\ead[label=e2, mark]{zhou@ese.eur.nl, C.Zhou@DNB.nl}}
%\and
%\author[B]{\fnms{Chen} \snm{Zhou}\ead[label=e3,mark]{???@???}}
%%%%%%%%%%%%%%%%%%%%%%%%%%%%%%%%%%%%%%%%%%%%%%
%% Addresses                                %%
%%%%%%%%%%%%%%%%%%%%%%%%%%%%%%%%%%%%%%%%%%%%%%
\address[A]{ Erasmus University Rotterdam, \printead{e1}}

\address[B]{Erasmus University Rotterdam, De Nederlandsche Bank, and Tinbergen Institute, \printead{e2}}
\end{aug}

\begin{abstract}
The block maxima (BM) approach in extreme value analysis fits a sample of block maxima to the Generalized Extreme Value (GEV) distribution. We consider all potential blocks from a sample, which leads to the All Block Maxima (ABM) estimator. Different from existing estimators based on the BM approach, the ABM estimator is permutation invariant. For i.i.d.\ observations with positive extreme value index $\gamma > 0$, we show the asymptotic normality of the ABM estimator, which has a lower asymptotic variance than the disjoint BM and sliding BM estimators. Simulation studies justify our asymptotic theories. A key step in establishing the asymptotic theory for the ABM estimator is to obtain asymptotic expansions for the tail empirical process based on higher order statistics with weights.
\end{abstract}

\begin{keyword}[class=MSC2020]
\kwd[Primary ]{62G32}
\kwd[; secondary ]{62G30}
\end{keyword}

\begin{keyword}
\kwd{Block maxima method}
\kwd{maximum likelihood estimation}
\kwd{weighted tail empirical process}
\kwd{Radon-Nikodym derivative}
\kwd{heavy-tails}
\end{keyword}

\end{frontmatter}

%%%%%%%%%%%%%%%%%%%%%%%%%%%%%%%%%%%%%%%%%%%%%%
%% Please use \tableofcontents for articles %%
%% with 50 pages and more                   %%
%%%%%%%%%%%%%%%%%%%%%%%%%%%%%%%%%%%%%%%%%%%%%%
%\tableofcontents

%%%%%%%%%%%%%%%%%%%%%%%%%%%%%%%%%%%%%%%%%%%%%%
%%%% Main text entry area:
\section{Introduction} \label{sec: intro}

Consider a random sample $\{ X_1, X_2,...,X_n \}$ with a common distribution function $F$ and assume that $F$ belongs to the domain of attraction of a Generalized Extreme Value (GEV) distribution: there exist some sequences $a_n>0$ and $b_n$ such that
\begin{align} \label{DoA}
\lim_{n \to \infty} F^{n}(a_nx+b_n) = G_{\gamma}(x)= \exp(-(1+x\gamma)^{-1/\gamma}), \text{ for } 1+\gamma x>0.
\end{align}
The only parameter $\gamma$ in the limit distribution is called the \textit{extreme value index}. It governs not only the limit distribution but also the tail behavior of the original distribution $F$. The case $\gamma > 0$ corresponds to heavy-tailed distributions (e.g., Pareto, Student-t), $\gamma = 0$ to light-tailed distributions (e.g., normal, exponential), and $\gamma < 0$ to distributions with a finite upper endpoint. Estimation of $\gamma$ is central to applications in finance, insurance, hydrology, and environmental science.

To estimate $\gamma$, the classical Block Maxima (BM) approach follows from the domain of attraction condition. One may divide the entire sample into $k$ disjoint blocks of size $ m=n/k$ and fit the corresponding block maxima to the GEV distribution, for example, by the Maximum Likelihood (ML) method. This results in an estimator of $(\gamma, a_m, b_m)$. Such estimators are consistent and jointly asymptotically normal under mild conditions \citep{bucher2016maximum, dombry2019maximum}. Besides constructing blocks disjointly, one may also construct blocks in various ways to increase the number of block maxima. For example, \cite{buecher2018} construct a sliding block estimator that is more efficient by considering the maxima of consecutive overlapping blocks, thereby introducing dependence between the block maxima even if the underlying observations are independent.

An undesired feature of all existing estimators from the BM approach is that they are not permutation invariant. If the observations are i.i.d., estimators of the parameter $\gamma$ should not depend on the order of the observations. However, for all existing BM estimators, based on either disjoint or sliding blocks, permutating the observations will in general not lead to the same estimate of $\gamma$. Based on this fact, \cite{mefleh2019permutation} construct a permutation bootstrap method for reducing the estimation uncertainty in estimators from the BM approach.

We revisit and develop a permutation-invariant estimator first considered by \cite[Chapter~5]{segers2001}, derived from a BM approach.
% that breaks down the underlying dependence structure between the observations completely. Breaking down the dependence structure between the observations is justified by the fact that $\gamma$ is a parameter belonging to the  \textit{marginal} stationary distribution $F$. 
Specifically, we consider all possible blocks of a fixed size $m$ that could have occured when sampling from the original observations. Denote $X_{1:n} \leq X_{2:n} \leq ... \leq X_{n:n}$ as the order statistics of the random sample. The lowest observation that can be a block maximum is $X_{m:n}$. Note that $X_{m:n}$ can be the block maximum in one block only, that is, when the lowest $m$ observations form the block. In general, $X_{n-i+1:n}$ is a block maxima in $\binom{n-i}{m-1}$ blocks for $i=1,2,...,n-m+1$.  Therefore, the multiset consisting of all block maxima
\begin{align*}
\left\{  \underbrace{X_{n:n},..., X_{n:n}}_{ \binom{n-1}{m-1}}, ... \underbrace{X_{n-i+1:n},..., X_{n-i+1:n}}_{ \binom{n-i}{m-1} } , ... , \underbrace{X_{m:n}}_{1 } \right\}
\end{align*}
can be viewed as a repeated sample of order statistics, where $X_{n-i+1:n}$ is repeated $\binom{n-i}{m-1}$ times for $1 \leq i \leq n-m+1$. Consequently, the multiset contains $\binom{n}{m}$ repeated observations. We fit this multiset of all block maxima to the GEV distribution by means of the ML method as if they were i.i.d. observations. This results in a permutation-invariant estimator of $(\gamma,a_m, b_m)$ --- the All Block Maxima (ABM) estimator. Note that the idea of considering statistics based on all blocks has been investigated in \cite[Chapter 5]{segers2001}. Nevertheless, only consistency has been studied therein, whereas we aim at establishing the asymptotic normality of the ABM estimator.

A heuristic rationale for the validity of the ABM estimator can be provided by considering an equivalent characterization of $F$ belonging to the domain of attraction of a GEV distribution: the limit relation \eqref{DoA} is equivalent to the existence of a positive function $f$ such that
\begin{align}\label{GPD}
\lim_{t \uparrow x^*} \frac{1-F(t+xf(t))}{1-F(t)}=D_\gamma(x):=(1+\gamma x)^{-1/\gamma}, \text{ for } 1+\gamma x>0.
\end{align}
where $x^* =\sup \{ x: F(x)<1 \}$ and $1-D_\gamma(x)$ is the CDF of the Generalized Pareto (GP) distribution \citep[Theorem 1.1.6]{deHaanFerreira2006extreme}. Consequently, order statistics above a high threshold can be approximately regarded as order statistics from a GP distribution. One can assign proper weights to these order statistics in line with the Radon-Nikodym (RN) derivative between the GEV and GP distribution so that the resulting weighted order statistics will approximately follow a GEV distribution. In Section \ref{subsec: measure} we show that the repetition scheme of the order statistics implicit in the multiset of all block maxima corresponds to the RN derivative between the GEV and GP distribution. Therefore, fitting the sample of all block maxima to the GEV distribution by ML is a valid estimation method.

Our main theorem, Theorem \ref{main_theorem}, states the asymptotic normality of the ABM estimator when the observations are i.i.d. and $\gamma>0$. The proof relies heavily on a weighted tail empirical process result of the order statistics, presented in Proposition \ref{empirical_distribution}. This result is of independent interest and can be used for proving asymptotic behaviour of other estimators based on all block maxima.

The ABM estimator has a lower asymptotic variance than the disjoint BM and sliding BM estimators. This theoretical result is confirmed by extensive simulation studies. Moreover, the simulation study further suggests that the ABM estimates against the effective number of observations used in estimation, $n/m$, yield a smooth path which facilitates a straightforward choice of the optimal number of blocks. Finally, the ABM estimator also performs better than the disjoint BM estimator when the observations are not i.i.d.: both for observations forming a serially dependent but stationary time series, and for observations drawn from non-stationary distributions.

%%----------------------------------------------------------------------------------------------------------------------------------------
\section{The All Block Maxima Method: $\gamma>0$} \label{sec: idea}

We study and apply the ABM estimator in case of a positive extreme value index. For $\gamma>0$, the domain of attraction condition implies that for $\sigma_n=F^{\leftarrow}(1-1/n)$ and any $x>0$, 
\begin{align}\label{Fr\'echet_DOA}
 \lim_{n \to \infty} F^{n} (\sigma_n x)  = \exp \left(  -x^{-1/\gamma} \right).
\end{align}
 Compared to the general domain of attraction condition in \eqref{DoA}, the shift $b_n$ is set to zero and the normalizing constant $a_n$ is related to $\sigma_n$ by $\sigma_n=\gamma a_n$.  Consequently, we can fit the sample of all block maxima $
\  \underbrace{X_{n:n},..., X_{n:n}}_{ \binom{n-1}{m-1}}, ... \underbrace{X_{n-i+1:n},..., X_{n-i+1:n}}_{ \binom{n-i}{m-1} } , ... , \underbrace{X_{m:n}}_{1 } 
$ to a scaled Fr\'{e}chet distribution to estimate $(\gamma,\sigma_n)$.
\subsection{Weighted Maximum Likelihood for estimating $\gamma$} \label{subsec: estimator}
Similar to \cite{bucher2018maximum}, we deal with potentially negative observations by left-truncating all the observations with a constant $c>0$. Denote $X_i^c=\max(X_i,c)$ and denote the CDF of the scaled Fr\'echet distribution by $G(x; \theta)=\exp\left(-\left(\frac{x}{\sigma}\right)^{-1/\gamma} \right)$ for $x>0$, with shape and scale parameters $\theta=( \gamma, \sigma   ) \in (0, \infty)^2$. Note that in the context of ML estimation, the log-likelihood based on a sample of repeated observations can be viewed as a weighted log-likelihood for non-repeated observations. When fitting the sample of all block maxima to the scaled Fr\'echet distribution, the weight corresponding to the order statistic $X_{n-i+1:n}$ is
\begin{align}\label{weight_pi}
p_i=\binom{n-i}{m-1}/\binom{n}{m}, \text{ with    } \sum_{i=1}^{n-m+1}p_i=1.
\end{align}
The log-likelihood function is then
\begin{align}\label{Frechet_likelihood}
L(\theta)&=\sum_{i=1}^{n-m+1} p_i l_\theta\left(X^c_{n-i+1:n}\right), \\
\mbox{ with } l_\theta(x)&= \log\left(\frac{1}{\gamma \sigma}\right)-\left(\frac{x}{\sigma}\right)^{-1/\gamma}-(1/\gamma+1)\log\left(\frac{x}{\sigma}\right).
\end{align}
By taking the partial derivatives of $L(\theta)$ with respect to $\theta=(\gamma, \sigma)$, we obtain that the ML estimator $\hat{\gamma}$ is given by the zero of the function
\begin{align*}
\Psi_n(\gamma)=\gamma +
\frac{\sum_{i=1}^{n-m+1}p_i \left( X^c_{n-i+1:n}\right)^{-1/\gamma}\log \left(X^c_{n-i+1:n}\right)}{\sum_{i=1}^{n-m+1}p_i \left( X^c_{n-i+1:n} \right)^{-1/\gamma}}-\sum_{i=1}^{n-m+1}p_i \log \left( X^c_{n-i+1:n} \right),
\end{align*}
and the ML estimator for $\sigma$ is given by
\begin{align*}
\hat{\sigma}=\sum_{i=1}^{n-m+1} \left( p_i \left( X^c_{n-i+1:n} \right) ^{-1/\hat{\gamma}} \right)^{-\hat{\gamma}}.
\end{align*}
The existence and uniqueness of the ML estimator is guaranteed if the order statistics do not all have the same value. This follows directly from Lemma 2.1 in \cite{bucher2018maximum}.

Finally, for practical purposes, instead of computing the binomial coefficients for each weight, it is more efficient to calculate the weights $p_i$ via the following recursion, initiated by $p_1=m/n$,
\begin{align*}
p_{i+1}=\left( \frac{n-m-i}{n-1-i} \right) p_i,  \hspace{10mm}  \text{ for }  1 \leq i \leq n-m+1.
\end{align*}
\subsection{Measure Transformation} \label{subsec: measure}
The ABM estimator can be heuristically understood by considering the measure transformation from a GP distribution to a GEV distribution. Here we present that heuristic argument in more detail for $\gamma>0$. 

First, for $\gamma>0$, the measure transformation can be simplified to that from a Pareto to a Fr\'echet distribution because the domain of attraction condition \eqref{Fr\'echet_DOA} is equivalent to
\begin{align}\label{Pareto}
\lim_{t \to \infty} \frac{1-F(tx)}{1-F(t)}:=H_{\gamma}(x)=x^{-1/\gamma} \text{ for all } x>0.
\end{align}
where $1-H_{\gamma}(x)$ is the CDF of a Pareto distribution with shape parameter $\gamma$. Effectively, for a large threshold $t$, excess ratios over the threshold follow a Pareto distribution. For example, take $t=X_{n-k:n}$, and denote  $Y_{k-i+1:k}=X_{n-i+1:n}/t$ for $i=1,2,\ldots,k$, then $Y_{1:k}\leq \ldots\leq Y_{k:k}$ can be approximately viewed as order statistics from the Pareto distribution. Heuristically, we have $H_{\gamma}(Y_{k-i+1:k}) \approx 1-\frac{k-i+1}{k}=\frac{i-1}{k}$ for $1 \leq i \leq k$.

Secondly, consider the Radon-Nikodym derivative between the Fr\'echet and Pareto distribution. The Pareto and the Fr\'echet distribution are probability measures with densities
$$d_{\text{Pareto}}(x):=\frac{1}{\gamma}x^{-1/\gamma -1}  \mbox{ and  } d_{\text{Fr\'echet}}(x):=d_{\text{Pareto}}(x) \exp(-H_{\gamma}(x))$$ respectively. Consequently, the Radon-Nikodym derivative between the Fr\'echet and the Pareto distribution is
$$\frac{d_{\text{Fr\'echet}}}{d_{\text{Pareto}}}(x)=\exp(-H_\gamma (x)).$$

Following all of the above, the RN derivative evaluated at the observation $Y_{k-i+1:k}$ satisfies $$ \frac{d_{\text{Fr\'echet}}}{d_{\text{Pareto}}}\left(Y_{k-i+1:k}\right)=\exp \left( -H_{\gamma}\left(Y_{k-i+1:k}\right) \right) \approx\exp \left( - \frac{i-1}{k}\right).$$ Moreover, notice that $\sum_{i=1}^{\infty} \exp \left( -\frac{i-1}{k} \right)= \frac{1}{1-e^{-1/k}} \approx k$. Therefore, if we assign the weight $$q_i=\frac{1}{k} \exp \left( - \frac{i-1}{k}\right)$$ to the order statistic $Y_{k-i+1:k}$ for $i \leq i \leq k$, the weighted order statistics will approximately follow a Fr\'echet distribution. Note that $q_i$ does not depend on $\gamma$ or any auxiliary function.

We will show in Lemma \ref{weights} that the weights $p_i=\binom{n-i}{m-1}/\binom{n}{m}$ used in the ABM method are uniformly close to the weights derived from the measure transformation $q_i$. Essentially, the ABM method, by changing the equal weight for the order statistics to $p_i$ in the likelihood \eqref{Frechet_likelihood}, is tranforming the weighted higher order statistics to approximately Fr\'echet distributed order statistics with a proper scale. 

\section{Asymptotic theory} \label{sec: asymptotic}
\subsection{Conditions} \label{subsec: conditions}
We assume the standard second order condition to characterize the speed at which the limit in Condition \eqref{DoA} is attained, see Theorem 2.3.9 in \cite{deHaanFerreira2006extreme}.
\begin{condition}\label{SOC_condition}
There exist an eventually positive or negative function $a$ with $\lim_{t \to \infty} a(t)=0$, $\gamma>0$ and  $\rho < 0$ such that
\begin{align}\label{SOC_condition}
\lim_{t \to \infty}  \frac{1}{a(t)} \left( \frac{1-F(tx)}{1-F(t)}-x^{-1/\gamma} \right)=x^{-1/\gamma} \frac{x^{\frac{\rho}{\gamma}}-1}{ \gamma \rho}, \hspace{5mm} \forall x>0.
\end{align}
\end{condition}

\begin{remark}
Because the ABM method uses the order statistics directly, the second order condition in \eqref{SOC_condition} corresponds to the second order condition used in the Peaks-Over-Threshold (POT) approach.  A similar situation arises in \cite{wager2014subsampling} where a subsampling maxima approach is proposed. The theoretical results therein also rely on a POT second order condition on the function $U=(1/(1-F) )^{\leftarrow}$, where $^\leftarrow$ is the left-continuous inverse.

For other BM methods, such as using disjoint or sliding blocks, usually one considers the function $V=(1/(-\log F) )^{\leftarrow}$ and assumes a second order condition for $V$. For instance, there exist an eventually negative or positive function $A(t)$ with $\lim_{t \to \infty} A(t)=0$, $\gamma>-1/2$ and $\rho'\leq 0$ such that
\begin{align*}
\lim_{t \to \infty} \frac{\frac{V(tx)-V(t)}{a(t)}-\frac{x^{\gamma}-1}{\gamma}}{A(t)} =\int_{1}^{x}s^{\gamma-1} \int_{1}^{s}u^{\rho-1}du ds,
\end{align*}
exists for all $x>0$, see \cite{dombry2019maximum}. When $\gamma>0$, as in \eqref{SOC_condition}, this condition can be simplified to
\begin{equation}\label{SOC_condition_V}
\lim_{t \to \infty} \frac{1}{A(t)}\left( \frac{V(tx)}{V(t)} -x^{\gamma} \right) =x^{\gamma} \frac{x^{\rho'}-1}{\rho'}=:\Psi_{\gamma, \rho' }(x),
\end{equation} for all $x>0$.

The two types of second order conditions in \eqref{SOC_condition} and \eqref{SOC_condition_V} are similar but not equivalent. They are equivalent if $-1 \leq \rho, \rho'<0$. In this case, $\rho=\rho'$.  In addition, under mild conditions, if the second order condition in \eqref{SOC_condition} holds, one obtains the second order condition in \eqref{SOC_condition_V}, with $\rho'=\max(\rho,-1)$, see \cite{drees2003large}.
\end{remark}

The second order condition \eqref{SOC_condition} implies the following inequality; see Theorem 5.1.4 in \cite{deHaanFerreira2006extreme}. For any $\epsilon, \delta>0$ there exists $t_0=t_0(\epsilon, \delta)>1$ such that for all $t$, $tx \geq t_0$,
\begin{align}\label{uniform_ineq}
\left| \frac{\frac{1-F(tx)}{1-F(t)}-x^{-1/\gamma}}{a(t)}-x^{-1/\gamma} \frac{x^{ \rho / \gamma}-1}{\rho \gamma} \right| \leq \epsilon x^{-1/\gamma+\rho/\gamma}\max(x^\delta, x^{-\delta})
\end{align}

Next we impose the following conditions on the number of blocks $k$ (or equivalently on the block size $m$).
\begin{condition}\label{n,k,m_condition}
Let $k:=k_n$ and $m:=n/k$ be sequences satisfying that, as $n \to \infty$
\begin{align*}
\frac{k}{n^{l'}} \to \infty   \text{ with } l'>0 \text{ and } k=O(n^l) \text{ with }  l<\frac{\max(\rho,-1)}{\max(\rho,-1)-1/2}.
\end{align*}
\end{condition}

\begin{remark}\label{remark_condition3.3}
One can show that the function $a(t)$ is $\rho/ \gamma$-regularly varying as $ t \to \infty$. Moreover, by the properties of a regularly varying function for any $\delta>0$ there exists $N(\delta)$ such that for all $n> N(\delta)$ one obtains that $a(\sigma_m) <\left(\frac{n}{k}\right)^{\rho+\delta}$. Hence the requirement $l<\frac{\rho}{\rho-1/2}$ implies that $\sqrt{k} a(\sigma_m) \to 0$ as $n \to \infty$. This assumes away the asymptotic bias; see Theorem \ref{main_theorem}.

 For other BM methods, the requirement on $k$ is usually related to the second order condition \eqref{SOC_condition_V} with $\sqrt{k}A(n/k)=O(1)$ as $n\to\infty$. Notice that $A(t)$ is $\rho'$-regularly varying, thus w.l.o.g. assuming $A(t)=Ct^{\rho'}$ for some constant $C$, the requirement on $k$ for other BM methods is equivalent to
 $$k=O(n^l) \text{ with } l\leq\frac{\rho'}{\rho'-1/2}.$$
Compared to the other BM methods, our requirement on $k$ is in line with the fact that $\rho'=\max(\rho,-1)$.
\end{remark}
\subsection{Main theorem} \label{subsec: theorem}
The following theorem shows the asymptotic normality of the ABM estimator for $\gamma>0$.
\begin{theorem}\label{main_theorem}
Under Conditions \ref{SOC_condition} and \ref{n,k,m_condition}, with probability tending to one there exists a unique maximizer $(\hat{\gamma}_n, \hat{\sigma}_n )$ of the Fr\'echet log-likelihood given by \eqref{Frechet_likelihood}. Moreover, denoting $\tau$ as the Euler-Mascheroni constant and $\Gamma$ as the Gamma function, we have that, as $n \to \infty$,
\begin{align*}
\sqrt{k}
\left( \frac{1}{\hat{\gamma}_n} -\frac{1}{\gamma},\hspace{1mm} \hat{\sigma}_n/\sigma -1  \right)^{T}\td 	\mathbb{N} (0, M\Sigma M'), \text{ with } M= \frac{6}{\pi^2} \begin{bmatrix} \gamma^{-2} & \gamma^{-1}(1-\tau) &  -\gamma^{-2} \\ \tau-1&-\gamma(\Gamma''(2)+1)&1-\tau \end{bmatrix},
\end{align*}
and
$$
\Sigma=
\begin{bmatrix}
\gamma^2 p &-&-\\
-\frac{\gamma}{2}(1-\tau +\log(2))&1/2&- \\
\gamma^2 ((3-\tau-\frac{\log(2)}{2})\log(2) -\frac{\pi^2}{12}  )&-\gamma \log(2)&\gamma^2 2 \log(2)
\end{bmatrix},
$$
where $p=1/2(\tau +\log(8)-1) -\frac{1}{12}(\pi^2-6(\tau+\log(2))+6(\tau+\log(2))^2) -(1-\tau-\log(2))+(2-\tau)(1-\tau)$ and the empty entries are defined by symmetry of the matrix.

\end{theorem}
By calculating $M \Sigma M'$ and applying the delta method, we obtain that as $n \to \infty$, $\sqrt{k} (\hat{\gamma}-\gamma)  \overset{d}{\to} \mathbb{N}(0, \gamma^2 a) $ where $a \approx 0.393 $. Notice that, given the same sequence of $k$, the Hill estimator has asymptotic variance $\gamma^2$, see e.g. \cite{de1998asymptotic},  the BM estimator based on disjoint blocks has asymptotic variance $0.608 \gamma^2$ \citep{bucher2018maximum} and the sliding block quasi-maximum likelihood estimator has asymptotic variance $0.494 \gamma^2$ \citep{buecher2018}. Among these three BM-type estimators, the ABM estimator has the lowest asymptotic variance.
\begin{remark}
The asymptotic variance does not depend on unknown parameters beyond $\gamma$ itself. In practice, one may estimate $\gamma$ consistently by any method (including ABM) and plug in the estimate to construct confidence intervals. We demonstrate the accuracy of the resulting normal approximation and confidence intervals in the simulation study (Section~\ref{sec: simulation}).
\end{remark}
\begin{remark}
Theorem~\ref{main_theorem} is established under the assumption of i.i.d.\ observations with $\gamma > 0$. These restrictions parallel those in competing results for BM-type estimators; for instance, \cite{bucher2016maximum} also require $\gamma > 0$ in parts of their analysis. Extending the asymptotic theory to $\gamma \leq 0$ or to serially dependent observations are natural directions for future work. Nevertheless, our simulation study in Section~\ref{sec: simulation} provides evidence that the ABM estimator performs well beyond the i.i.d.\ setting.
\end{remark}
To prove Theorem \ref{main_theorem}, we establish the tail empirical process for the sample of all block maxima in the proposition below. Define the empirical tail distribution function of the truncated order statistics $X^c_{n-i+1:n}=\max \left( X_{n-i+1:n}, c \right) $ with weights $p_i$, given in \eqref{weight_pi}, as follows:
\begin{align*}
\mathbb{P}_n(x)=\sum_{i=1}^{n-m+1}p_i \mathbb{I}\{\frac{ X^c_{n-i+1:n} }{\sigma_m } > x\}.
\end{align*}
\begin{proposition}\label{empirical_distribution}
Under Conditions \ref{SOC_condition} and \ref{n,k,m_condition}, we have that for any $\delta>0$ and any $0<\epsilon<1/2$ there exists $\tilde{\epsilon}=\tilde{\epsilon}(\rho,l, \epsilon)$ and a series of standard Brownian motions $\{W_n(\cdot)\}_{n=1}^{\infty}$ defined on $[0,+\infty)$, such that as $n \to \infty$,
\begin{align*}
\sup_{x> \left( \log k \right)^{-\delta}}k^{\tilde{\epsilon}}x^{\frac{1/2-\epsilon}{\gamma}} \left| \sqrt{k} \left( \mathbb{P}_n(x) -\left(1-e^{-x^{-1/\gamma}} \right) \right) -e^{-x^{-1/\gamma}}W_n(x^{-1/\gamma})  \right|=o_p(1).
\end{align*}
\end{proposition}
All proofs are postponed to Section \ref{Proofs}.
%%----------------------------------------------------------------------------------------------------------------------------------------

\section{Simulation} \label{sec: simulation}

In this section, we examine the finite sample performance of the ABM estimator under both i.i.d.\ and non-i.i.d.\ settings. Throughout, we compare the ABM estimator with the disjoint block maxima (BM) estimator and the sliding block maxima (SBM) estimator of \cite{buecher2018}. All three estimators are computed by fitting a Fr\'{e}chet distribution via maximum likelihood: for BM, the likelihood is based on the $k$ disjoint block maxima with equal weights; for SBM, the likelihood is based on the $n-m+1$ sliding block maxima with equal weights; for ABM, the likelihood uses the weights $p_i$ given in \eqref{weight_pi}. The number of replications is $N=1000$ unless specified otherwise. In all figures, the ABM estimator is shown as a solid red line, the BM estimator as a dashed blue line, and the SBM estimator as a dotted green line.

\subsection{I.I.D.\ Data Generating Processes}

Let $X_1, \ldots, X_n$ be a random sample drawn from either the Pareto or Fr\'{e}chet distribution with shape parameter $\gamma$, the positive half of the Student-t distribution with $\nu$ degrees of freedom, or the Burr distribution. In Table~\ref{SOC_parameters} we summarize the first- and second-order parameters corresponding to these distributions.

\begin{table}[h!]
\centering
\caption{First- and second-order parameters}
\label{SOC_parameters}
\setlength{\tabcolsep}{4pt}
\renewcommand{\arraystretch}{1.5}
\small
\begin{tabular}{l c c c c}
\hline
 & Fr\'{e}chet & Student-t$(\nu)$ & Pareto & Burr$(\gamma, \rho)$ \\
\hline
$\boldsymbol{\gamma}$ & $\gamma$ & $1/\nu$ & $\gamma$ & $\gamma$ \\
$\boldsymbol{\rho}$   & $-1$     & $-2/\nu$ & $-\infty$ & $\rho$ \\
$\boldsymbol{\rho}'$  & $-\infty$ & $-2/\nu$ & $-1$ & $\max(\rho, -1)$ \\
\hline
\end{tabular}
\begin{flushleft}
{\footnotesize \emph{Note: $\boldsymbol{\rho}$ and $\boldsymbol{\rho}'$ denote the second-order parameters corresponding to conditions \eqref{SOC_condition} and \eqref{SOC_condition_V}, respectively. The Burr distribution has CDF $F(x) = 1-(1+x^{-\rho/\gamma})^{1/\rho}$ for $x>0$, allowing $\gamma$ and $\rho$ to be varied independently.}}
\end{flushleft}
\end{table}

\subsubsection{Asymptotic Variance}

First, we verify the asymptotic theory in Theorem~\ref{main_theorem}. The theorem implies that $\sqrt{k}(\hat{\gamma}_n - \gamma) \overset{d}{\to} \mathbb{N}(0, \gamma^2 a)$ with $a \approx 0.393$. We simulate $N=1000$ samples of $n$ observations from the Student-t$(2)$ distribution with $n$ ranging from $500$ to $10000$. For each sample, we compute the ABM estimator using $k = n^l$ with $l \in \{1/3, 1/2, 2/3\}$. The implied asymptotic variance, defined as $k \cdot \widehat{\mathrm{Var}}(\hat{\gamma}_n) / \gamma^2$, is plotted against $n$ in Figure~\ref{fig:asymptotic_variance}. The implied asymptotic variance fluctuates closely around the theoretical value $a \approx 0.393$ for all three choices of $k$, confirming the asymptotic theory.

\begin{figure}[h!]
\centering
\includegraphics[width=0.7\textwidth]{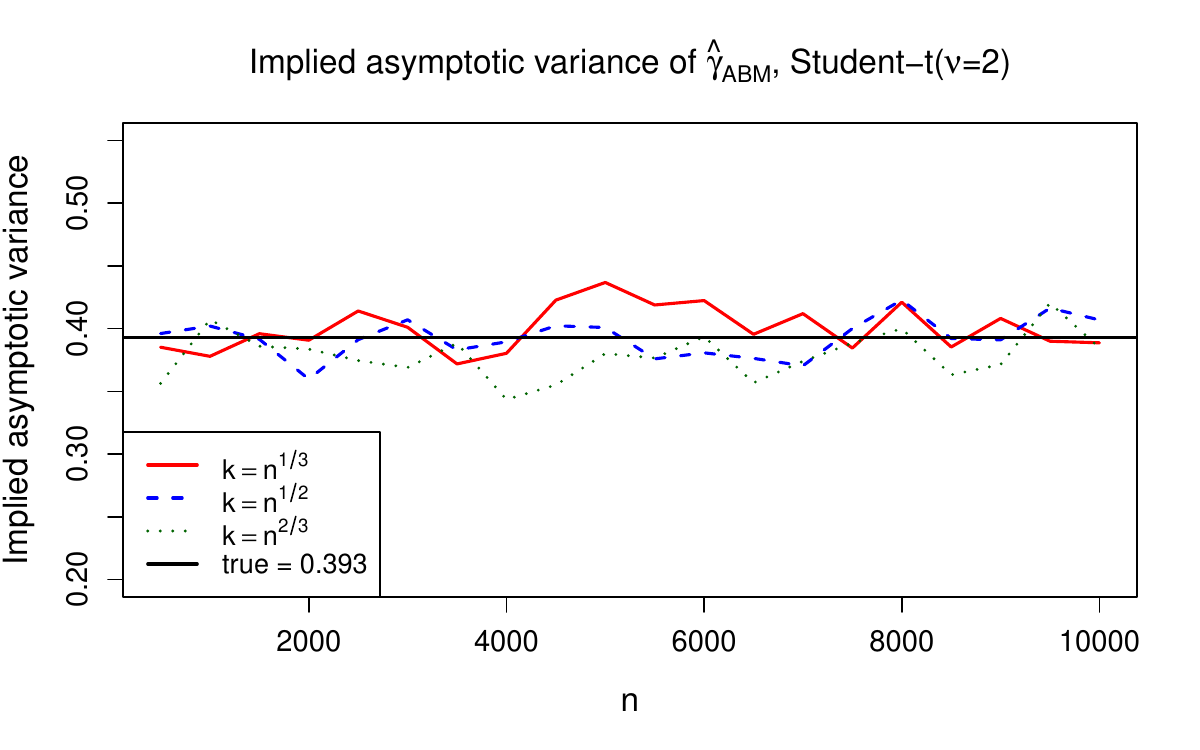}
\caption{Implied asymptotic variance of $\hat{\gamma}_{\mathrm{ABM}}$, Student-t$(\nu=2)$.}
\label{fig:asymptotic_variance}
\end{figure}

\subsubsection{Normal Approximation}

To assess the accuracy of the normal approximation implied by Theorem~\ref{main_theorem}, we examine both the distributional shape and the coverage of confidence intervals. We generate $N=1000$ samples of size $n=2000$ from the Student-t$(2)$ distribution and compute the ABM estimator with $k = \lfloor n^{1/3} \rfloor$. Figure~\ref{fig:normal_approx}(a) shows a QQ-plot of the standardized estimates $(\hat{\gamma}_n - \gamma)/(\gamma\sqrt{a/k})$ against the standard normal distribution. The points align closely with the theoretical line, with only mild departure in the upper tail. Figure~\ref{fig:normal_approx}(b) displays the empirical coverage of 90\% and 95\% confidence intervals, constructed using the asymptotic variance, across varying sample sizes $n \in \{500, 1000, 2000, 5000, 10000\}$. The coverage is close to the nominal level for all sample sizes, demonstrating that the normal approximation provides a reliable basis for inference.

\begin{figure}[h!]
\centering
\includegraphics[width=0.85\textwidth]{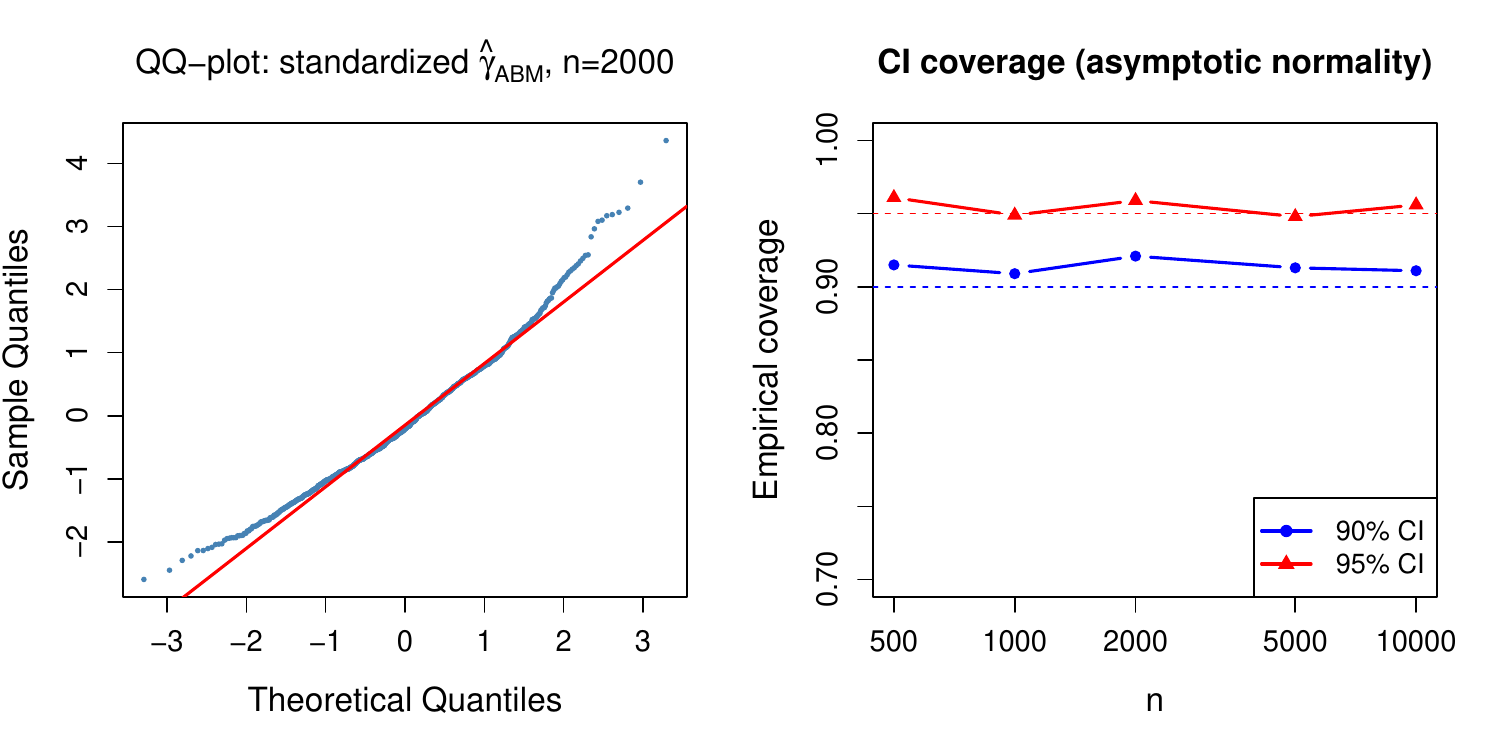}
\caption{Normal approximation accuracy of $\hat{\gamma}_{\mathrm{ABM}}$, Student-t$(\nu=2)$. (a)~QQ-plot of standardized estimates against $\mathbb{N}(0,1)$ for $n=2000$. (b)~Empirical coverage of confidence intervals based on the asymptotic normality across varying $n$.}
\label{fig:normal_approx}
\end{figure}

\subsubsection{Single Sample Performance}

Figure~\ref{fig:single_sample} compares the estimates from a single sample of $n=10000$ observations from the Student-t$(2)$ distribution against various values of $k$. All three estimators exhibit bias that increases with $k$. However, the ABM estimator demonstrates a noticeably smoother path as a function of $k$ compared to both the BM and SBM estimators. This smooth path facilitates a straightforward visual choice of the optimal~$k$ in applications.

\begin{figure}[h!]
\centering
\includegraphics[width=0.7\textwidth]{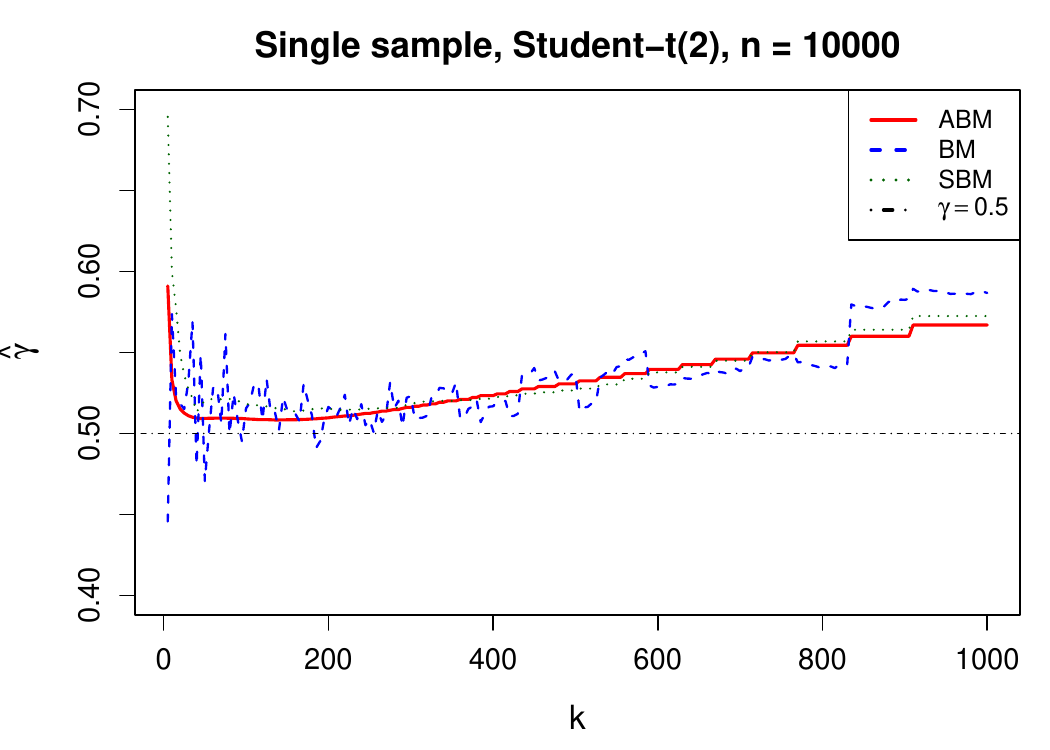}
\caption{Single sample performance, Student-t$(\nu=2)$, $n=10000$.}
\label{fig:single_sample}
\end{figure}

\subsubsection{MSE Comparison: Student-t}

We compare the MSE of the three estimators for the Student-t distribution with $\nu = 2, 3, 4, 5$ and sample size $n=1000$. The results are displayed in Figure~\ref{fig:mse_studentt}. For $\nu=2$, the ABM estimator clearly outperforms both BM and SBM across the entire range of $k$, with SBM performing between ABM and BM. As $\nu$ increases (lighter tails, $\rho = -2/\nu$ closer to zero), the advantage of ABM diminishes, and the three estimators become nearly indistinguishable. This is consistent with the fact that a smaller $|\rho|$ implies higher bias, and since ABM involves potentially lower order statistics, its bias is slightly larger. Nevertheless, with the variance reduction effect dominating the MSE, ABM still performs at par with or better than the competitors.

\begin{figure}[h!]
\centering
\subfloat[$\nu=2$]{\includegraphics[width=0.49\textwidth]{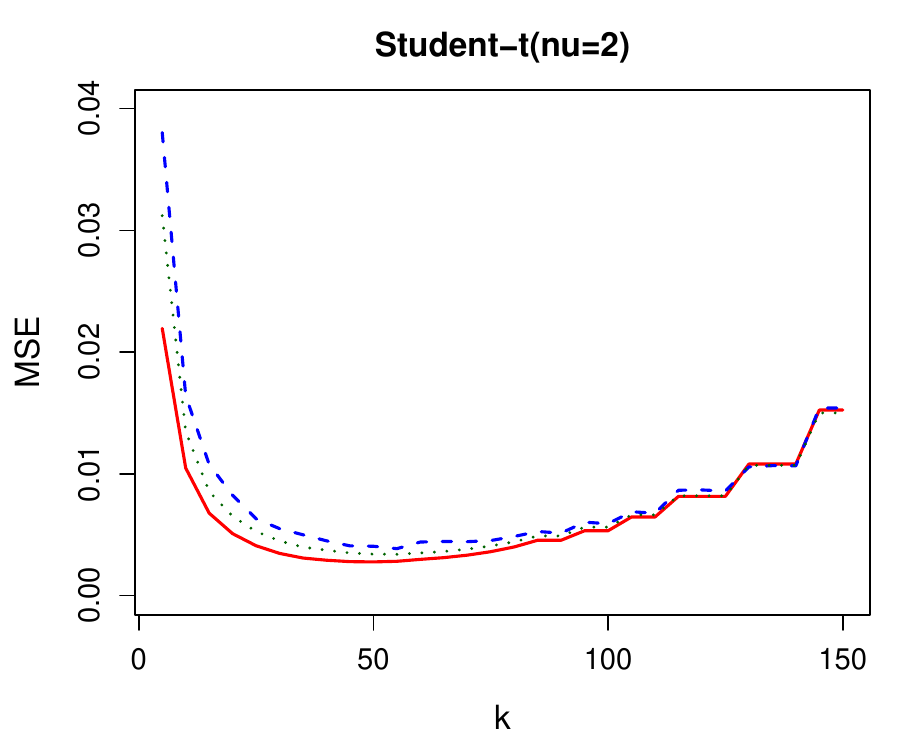}}\quad
\subfloat[$\nu=3$]{\includegraphics[width=0.49\textwidth]{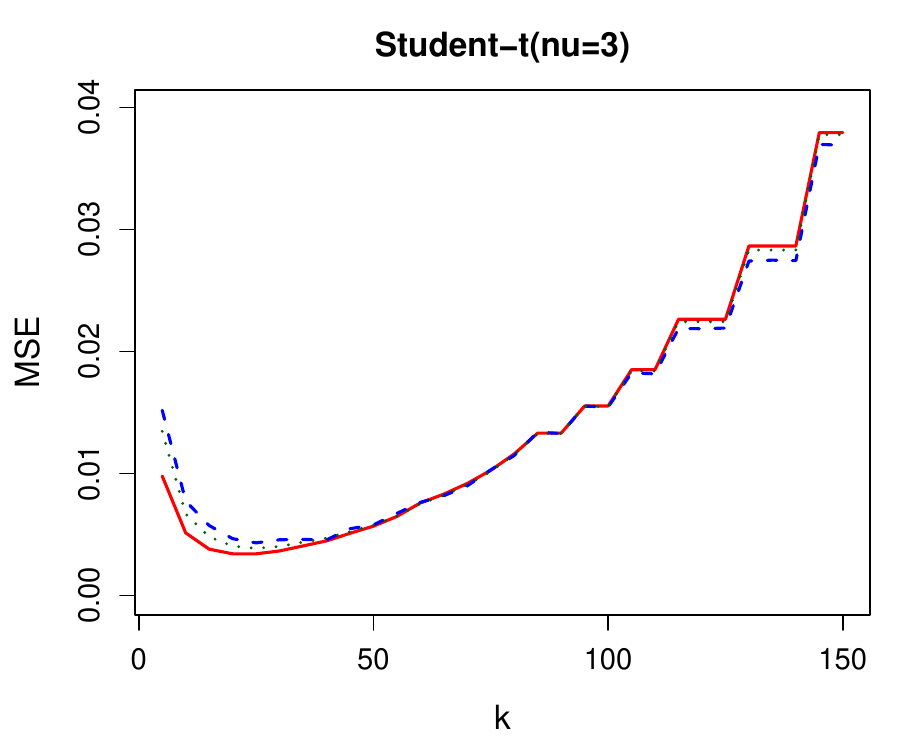}}\\
\subfloat[$\nu=4$]{\includegraphics[width=0.49\textwidth]{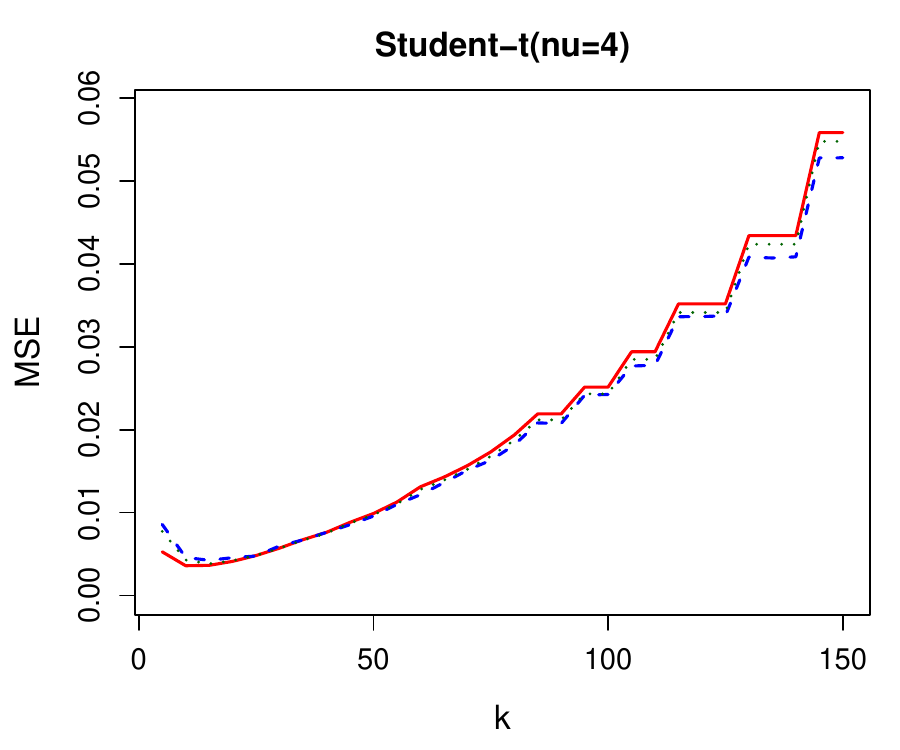}}\quad
\subfloat[$\nu=5$]{\includegraphics[width=0.49\textwidth]{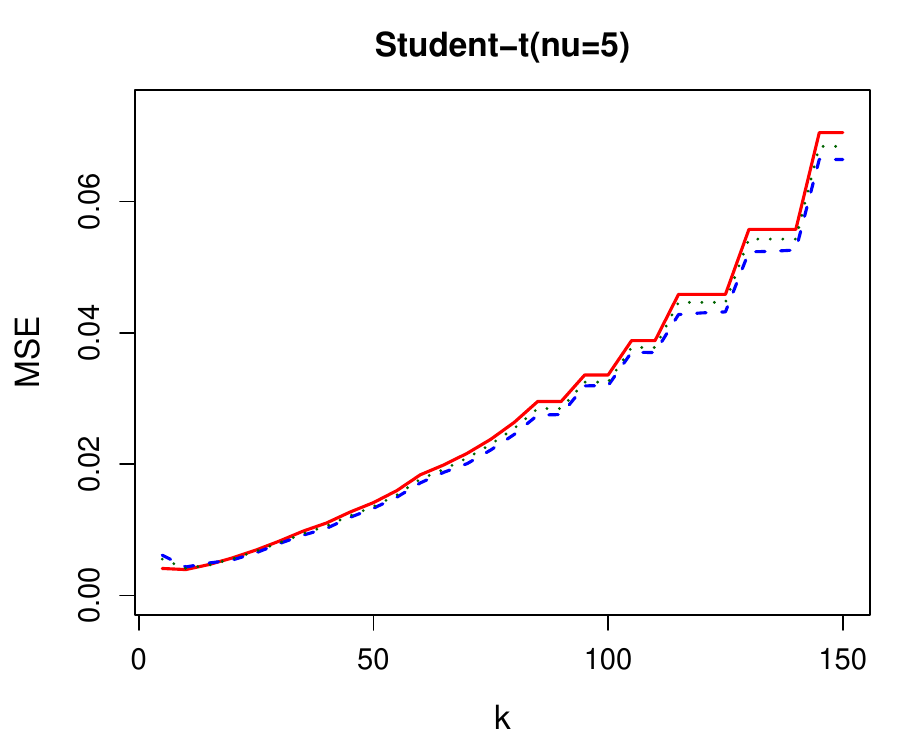}}
\caption{MSE for ABM (solid red), BM (dashed blue), and SBM (dotted green) estimators. Student-t distribution, $n=1000$.}
\label{fig:mse_studentt}
\end{figure}

\subsubsection{Bias-Variance Decomposition}

To understand the source of ABM's advantage, we decompose the MSE into bias and variance for the Student-t$(2)$ case in Figure~\ref{fig:bias_variance}. Regarding bias, the BM estimator has the smallest bias for small $k$, followed by SBM, with ABM exhibiting the largest bias. This is because ABM assigns weight to lower order statistics, which introduces additional bias. However, the ABM estimator has a substantially lower variance than both competitors, particularly for small to moderate $k$. Since the variance differences are an order of magnitude larger than the squared bias differences, the MSE is dominated by the variance component, explaining the superior MSE performance of the ABM estimator.

\begin{figure}[h!]
\centering
\subfloat[Bias]{\includegraphics[width=0.49\textwidth]{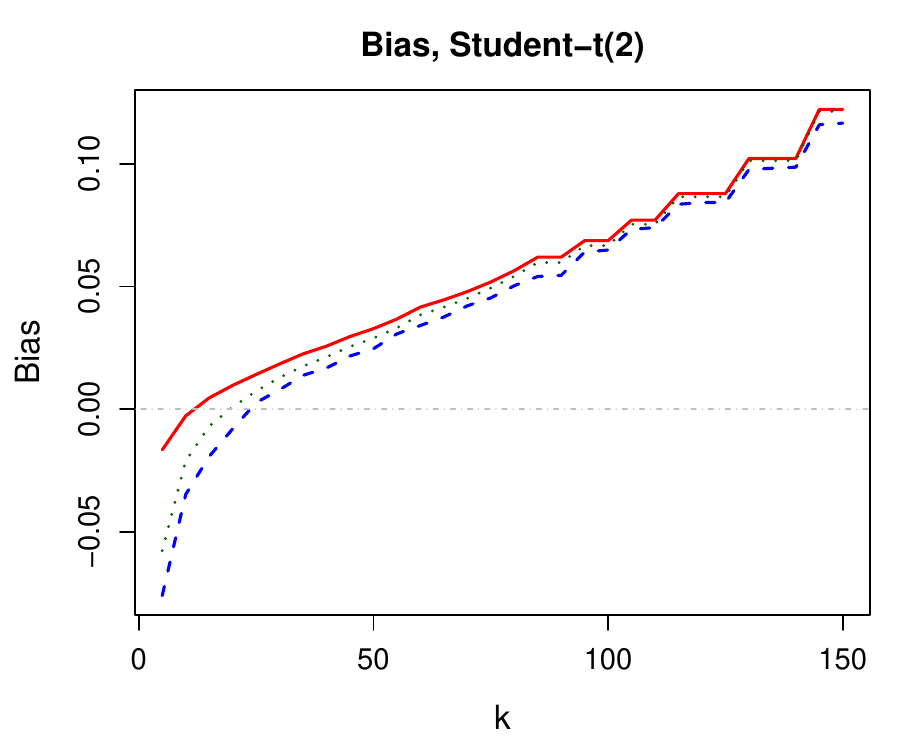}}\quad
\subfloat[Variance]{\includegraphics[width=0.49\textwidth]{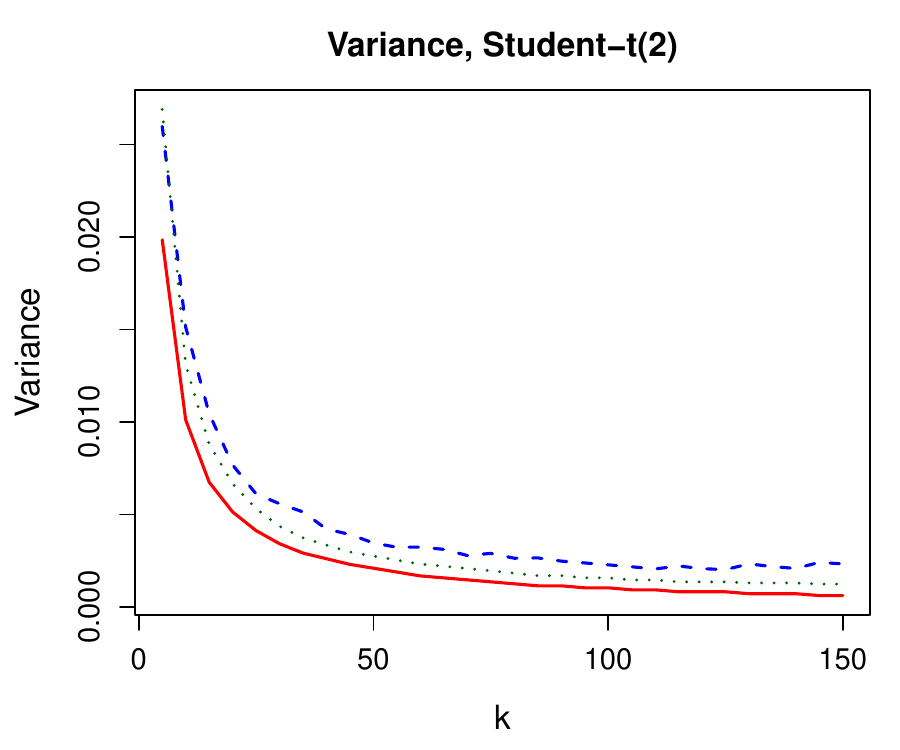}}
\caption{Bias and variance decomposition for ABM (solid red), BM (dashed blue), and SBM (dotted green). Student-t$(\nu=2)$, $n=1000$.}
\label{fig:bias_variance}
\end{figure}

\subsubsection{Separating $\gamma$ and $\rho$: Burr Distribution}

The Student-t distribution satisfies $\rho = -2\gamma$, so the effects of $\gamma$ and $\rho$ on estimation accuracy cannot be disentangled. To address this, we use the Burr distribution, defined by $F(x) = 1 - (1 + x^{-\rho/\gamma})^{1/\rho}$ for $x > 0$, which allows $\gamma$ and $\rho$ to be varied independently.

Figure~\ref{fig:burr} presents MSE curves for six configurations. The top row fixes $\gamma = 0.5$ and varies $\rho \in \{-0.5, -1, -2\}$. When $\rho = -0.5$ (slow convergence to the extreme value limit), all three estimators perform similarly and the MSE is large. For $\rho = -1$, a clear separation emerges: ABM has the lowest MSE, followed by SBM and then BM. For $\rho = -2$ (fast convergence), the MSE is dominated by variance and all three estimators are close, with ABM maintaining a slight edge. The bottom row fixes $\rho = -1$ and varies $\gamma \in \{0.25, 0.5, 1\}$. ABM consistently outperforms BM and SBM across all values of $\gamma$. These results confirm that ABM's advantage is a variance-reduction effect that is robust across different $\gamma$ and $\rho$ combinations, and is most pronounced when $\rho$ is moderate (neither too close to zero nor too negative).

\begin{figure}[h!]
\centering
\subfloat[$\gamma=0.5$, $\rho=-0.5$]{\includegraphics[width=0.32\textwidth]{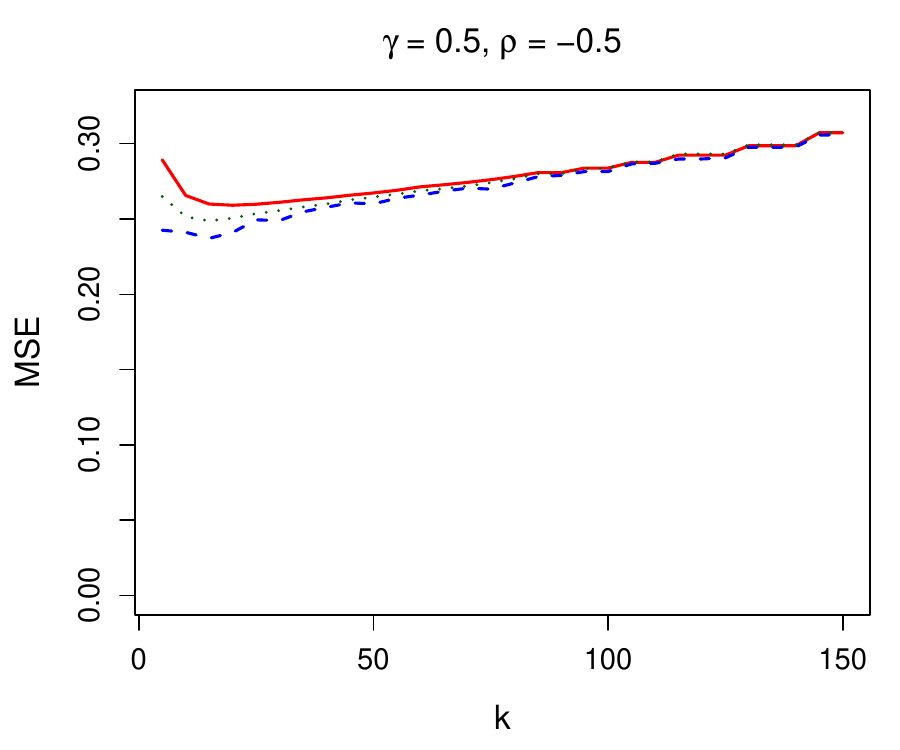}}\;
\subfloat[$\gamma=0.5$, $\rho=-1$]{\includegraphics[width=0.32\textwidth]{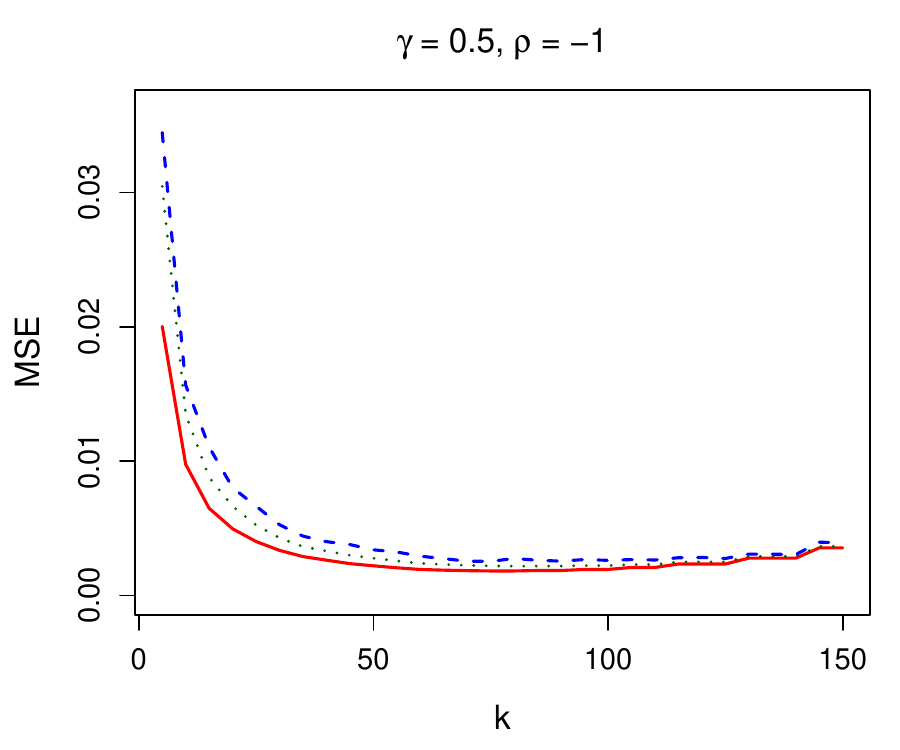}}\;
\subfloat[$\gamma=0.5$, $\rho=-2$]{\includegraphics[width=0.32\textwidth]{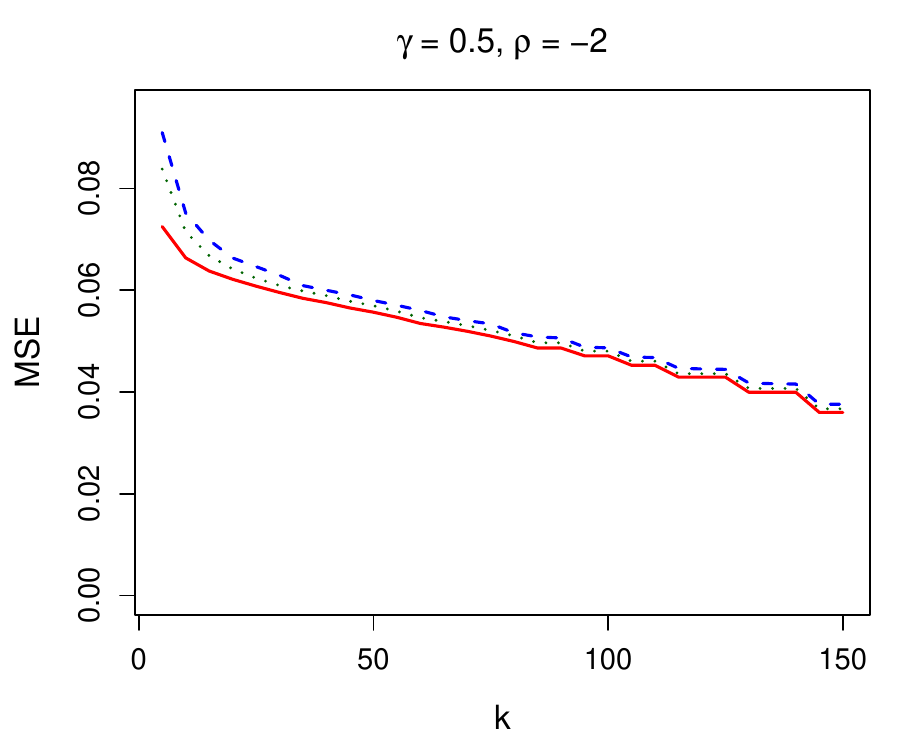}}\\
\subfloat[$\gamma=0.25$, $\rho=-1$]{\includegraphics[width=0.32\textwidth]{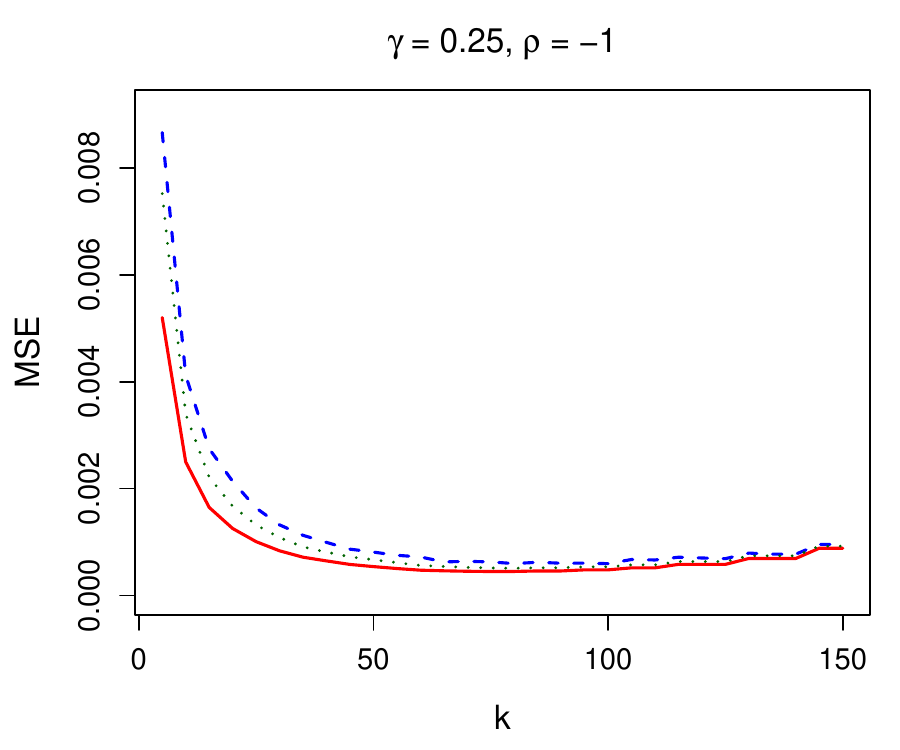}}\;
\subfloat[$\gamma=0.5$, $\rho=-1$]{\includegraphics[width=0.32\textwidth]{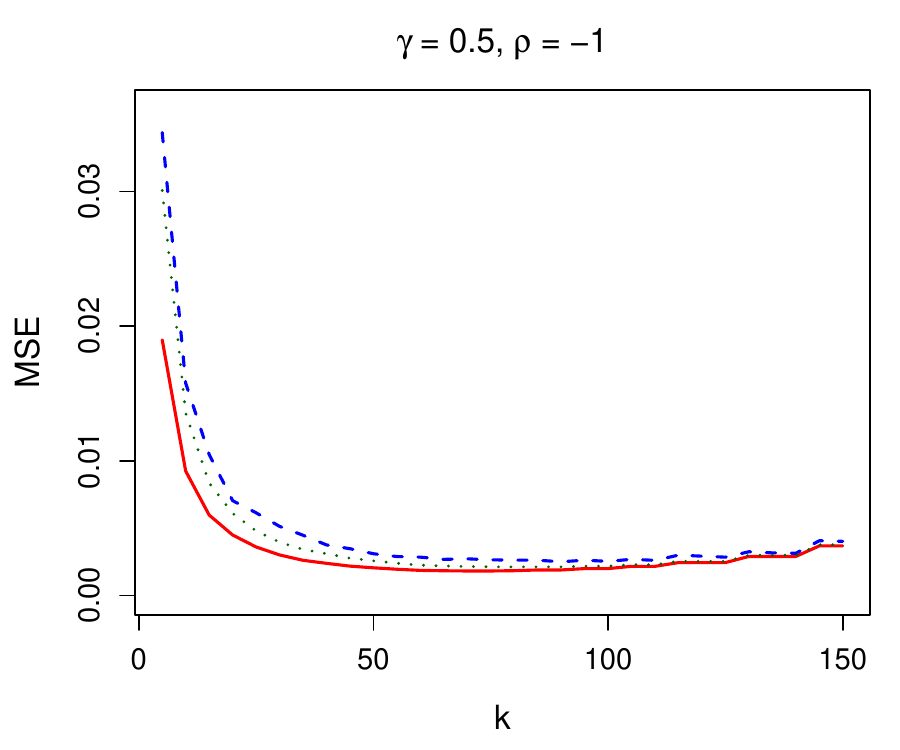}}\;
\subfloat[$\gamma=1$, $\rho=-1$]{\includegraphics[width=0.32\textwidth]{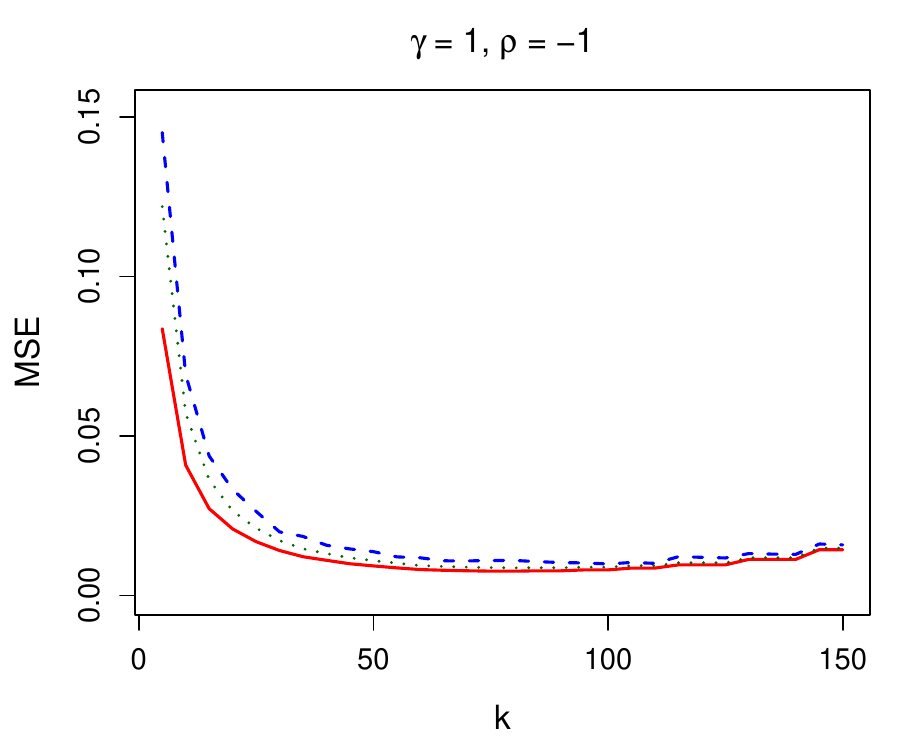}}
\caption{MSE for the Burr distribution. Top row: fixed $\gamma=0.5$, varying $\rho$. Bottom row: fixed $\rho=-1$, varying $\gamma$. Line styles as in Figure~\ref{fig:mse_studentt}.}
\label{fig:burr}
\end{figure}

\subsubsection{Fr\'{e}chet and Pareto Distributions}

Figure~\ref{fig:frechet_pareto} shows the MSE for the Fr\'{e}chet and Pareto distributions with $\gamma = 1/2$ and $n = 1000$. These distributions illustrate the case where the two second-order parameters $\rho$ and $\rho'$ diverge: for the Fr\'{e}chet distribution, $\rho = -1$ and $\rho' = -\infty$, while for the Pareto distribution, $\rho = -\infty$ and $\rho' = -1$. The ABM estimator performs at par with or better than BM and SBM across the entire range of $k$ for both distributions.

\begin{figure}[h!]
\centering
\subfloat[Fr\'{e}chet $(\gamma=1/2)$]{\includegraphics[width=0.49\textwidth]{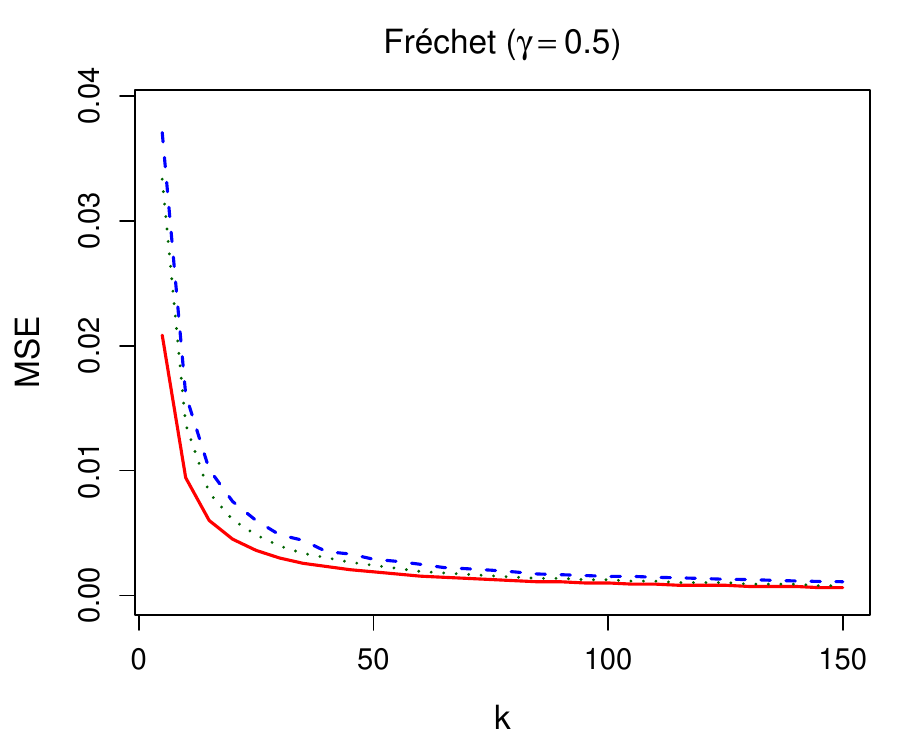}}\quad
\subfloat[Pareto $(\gamma=1/2)$]{\includegraphics[width=0.49\textwidth]{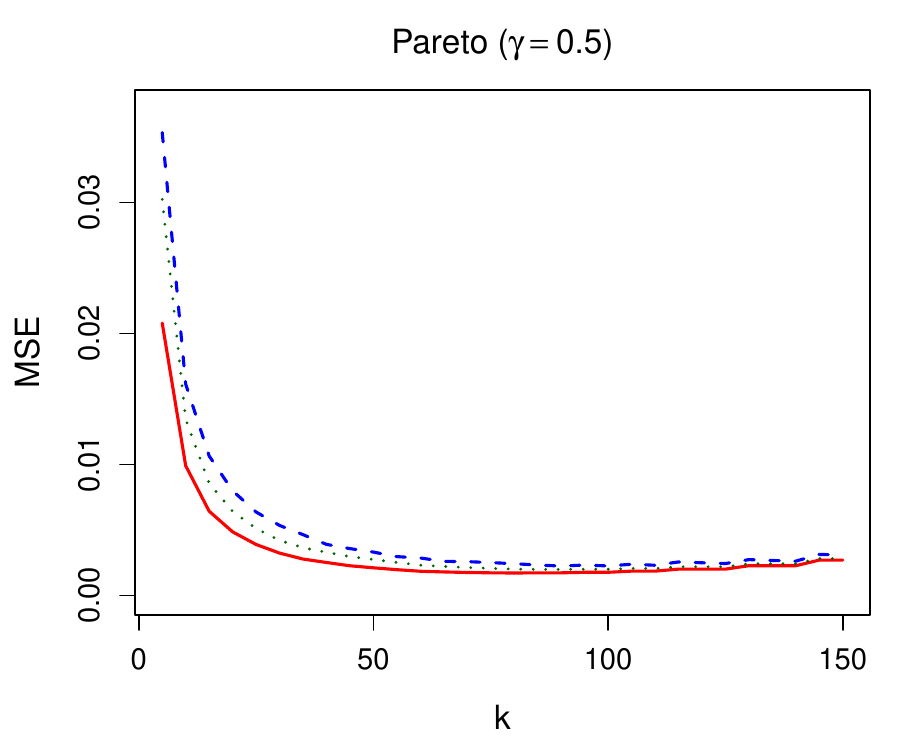}}
\caption{MSE for Fr\'{e}chet and Pareto distributions, $n=1000$. Line styles as in Figure~\ref{fig:mse_studentt}.}
\label{fig:frechet_pareto}
\end{figure}

\subsection{Non-I.I.D.\ Data Generating Processes}

Although Theorem~\ref{main_theorem} is proved for i.i.d.\ observations, we now explore the behaviour of the ABM estimator when the observations are serially dependent or not identically distributed.

\subsubsection{Linear Serial Dependence}

We introduce serial dependence through the AR(1) model
\begin{align}\label{model_dependence}
X_t = \phi X_{t-1} + \epsilon_t, \quad t \geq 2,
\end{align}
where $X_1 = \epsilon_1$ and $\{\epsilon_t\}$ is an i.i.d.\ sequence following the Student-t$(2)$ distribution. The parameter $\phi$ controls the degree of serial dependence. For this model, the extremal index is $\theta = 1 - \phi^4$ \citep{chernick1991calculating}. We consider $\phi \in \{0.1, 0.5, 0.9\}$, generate $2100$ observations, discard the first $100$, and use $n = 2000$ with truncation constant $c = 10^{-3}$.

Figure~\ref{fig:ar1} displays the MSE for the three estimators. For $\phi = 0.1$ (weak dependence), all estimators perform well at the optimal $k$, with ABM showing the smallest MSE. For $\phi = 0.5$, the BM and SBM estimators degrade rapidly for moderate to large $k$, while ABM remains stable over a wide range. For $\phi = 0.9$ (strong dependence), BM and SBM deteriorate severely even for small $k$, whereas the ABM estimator maintains moderate MSE throughout. This robustness of ABM under serial dependence can be explained by its permutation invariance: by discarding the order of the observations, ABM effectively treats the sample as if drawn from the stationary marginal distribution, which has the same extreme value index $\gamma$.

\begin{figure}[h!]
\centering
\subfloat[$\phi=0.1$]{\includegraphics[width=0.32\textwidth]{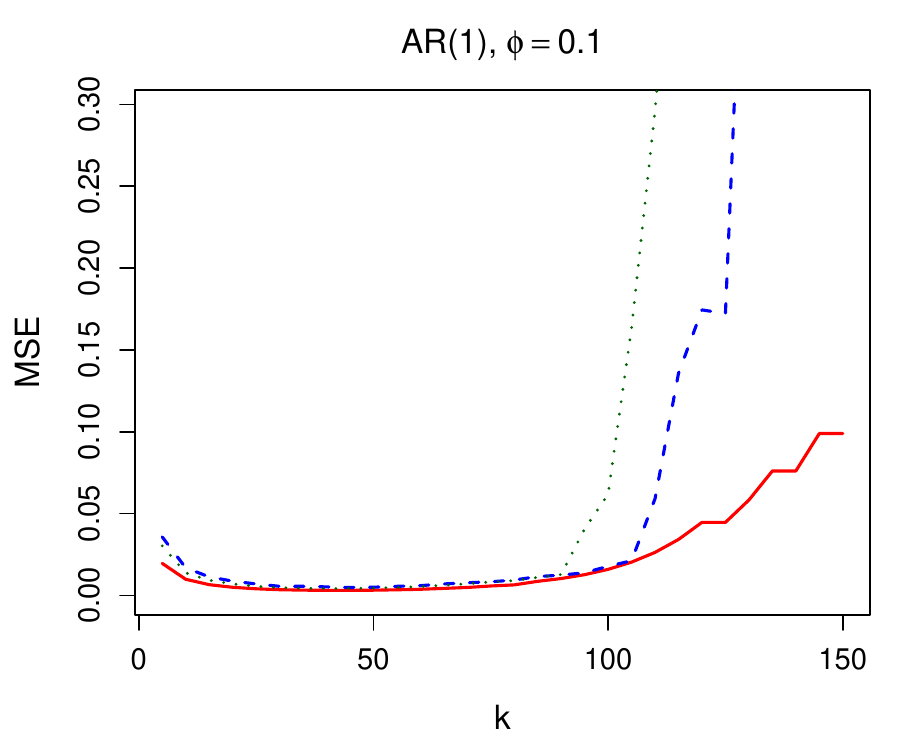}}\;
\subfloat[$\phi=0.5$]{\includegraphics[width=0.32\textwidth]{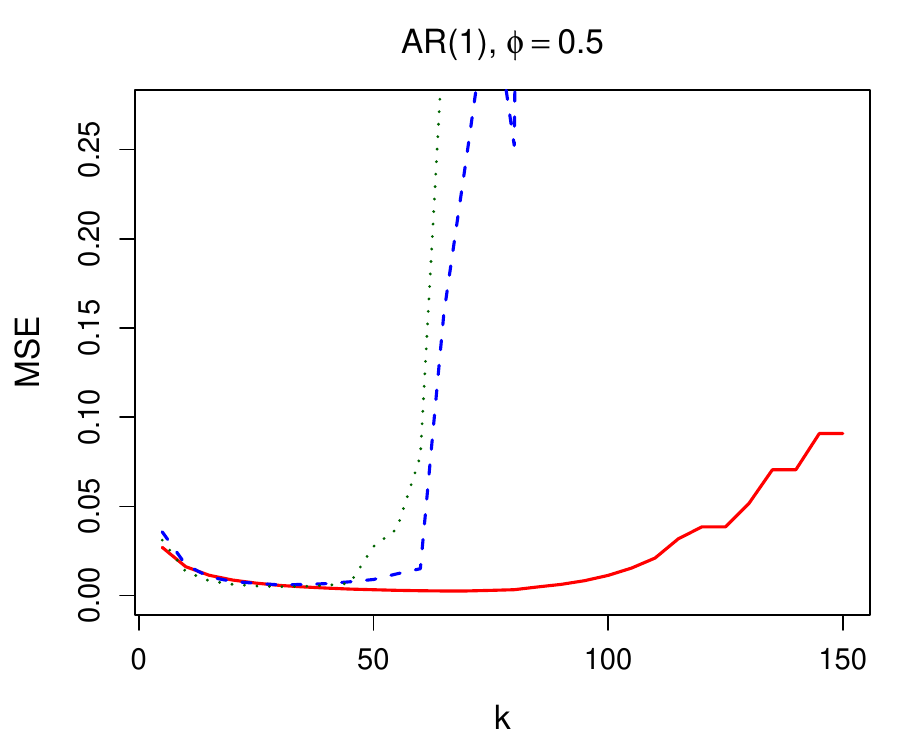}}\;
\subfloat[$\phi=0.9$]{\includegraphics[width=0.32\textwidth]{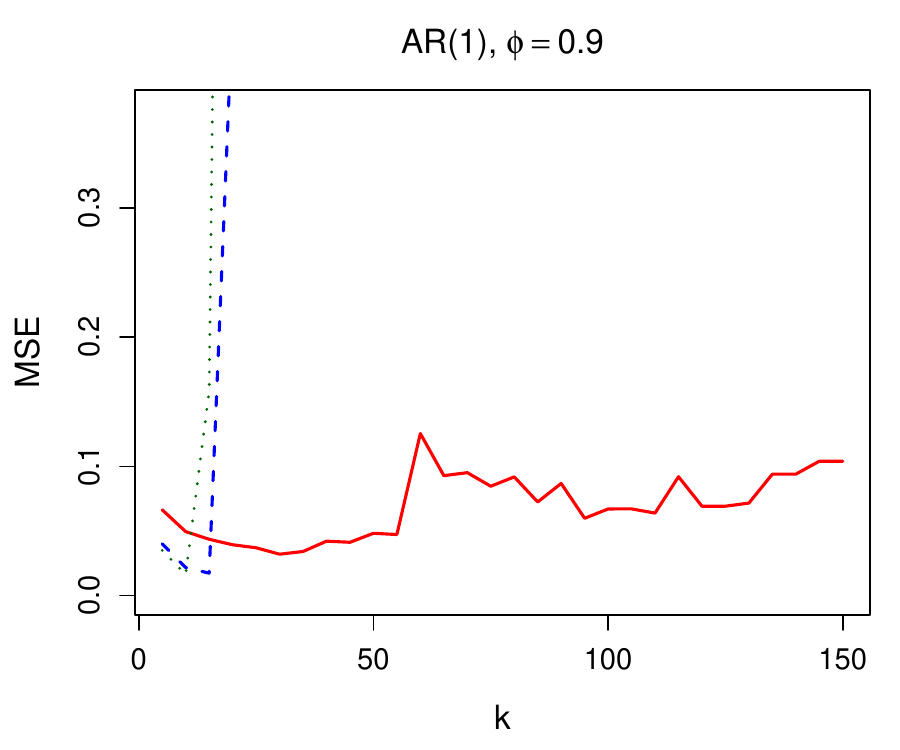}}
\caption{MSE under AR(1) serial dependence, $n=2000$. Line styles as in Figure~\ref{fig:mse_studentt}.}
\label{fig:ar1}
\end{figure}

\subsubsection{GARCH Dependence}

We simulate observations from a GARCH$(1,1)$ model:
\begin{align*}
X_t = \sigma_t \epsilon_t, \quad \sigma_t^2 = \lambda_0 + \lambda_1 X_{t-1}^2 + \lambda_2 \sigma_{t-1}^2,
\end{align*}
where $\lambda_1 + \lambda_2 < 1$ for stationarity and $\epsilon_t$ is i.i.d.\ standardized Student-t$(6)$ (with unit variance). We generate $2100$ observations, discard the first $100$, and use $n = 2000$ with $c = 10^{-3}$. Following \cite{kesten1973random}, the tail index satisfies $\gamma = 1/(2\kappa)$ where $\kappa$ solves $\E(\lambda_1 \epsilon_t^2 + \lambda_2)^\kappa = 1$. With $\lambda_0 = 0.5$, the combinations $\{\lambda_1 = 0.08, \lambda_2 = 0.91\}$ and $\{\lambda_1 = 0.11, \lambda_2 = 0.88\}$ yield $\gamma \approx 0.29$ and $\gamma \approx 0.35$, respectively.

Figure~\ref{fig:garch} shows the MSE for $k \leq 80$. For both parameter configurations, the ABM estimator achieves the lowest MSE at the optimal $k$ (around $k = 15$--$30$). For larger $k$, all estimators degrade due to bias, but ABM degrades more quickly because it assigns weight to order statistics that are affected by the time-varying volatility. We note that, as in the AR(1) case, the ABM estimator for $\gamma$ remains valid under serial dependence because $\gamma$ is a property of the stationary marginal distribution. However, the ABM scale estimator $\hat{\sigma}_n$ is not a valid estimator for the normalization constant of block maxima under dependence; estimating return levels would require a separate extremal index correction.

\begin{figure}[h!]
\centering
\subfloat[$\lambda_1=0.08$, $\lambda_2=0.91$]{\includegraphics[width=0.49\textwidth]{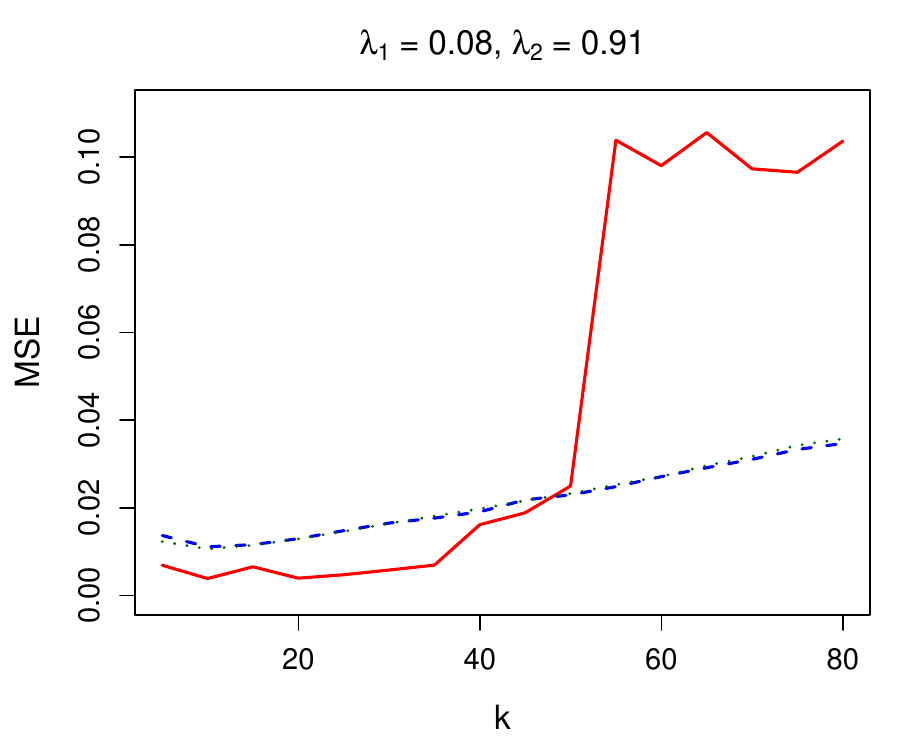}}\quad
\subfloat[$\lambda_1=0.11$, $\lambda_2=0.88$]{\includegraphics[width=0.49\textwidth]{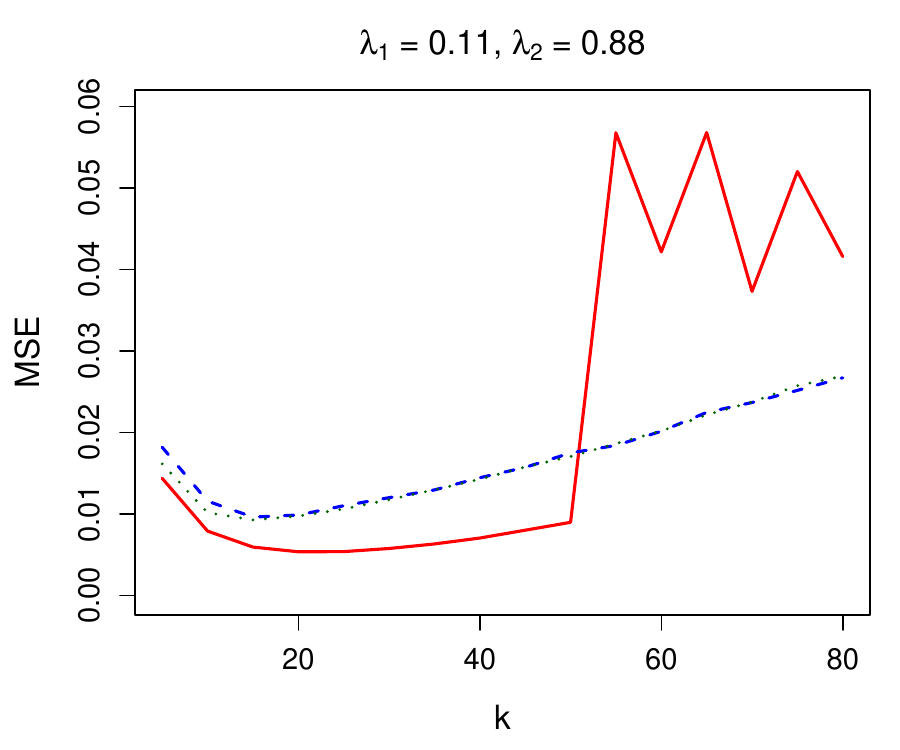}}
\caption{MSE under GARCH$(1,1)$ dependence, $n=2000$, $k \leq 80$. Line styles as in Figure~\ref{fig:mse_studentt}.}
\label{fig:garch}
\end{figure}

\subsubsection{Scale Heterogeneity}

We consider two models of scale heterogeneity with $\{\epsilon_t\}$ i.i.d.\ Student-t$(2)$ and $n = 1000$. The \emph{block heterogeneity} model assigns scale $r$ to the second half of the sample:
\begin{align*}
X_t = \begin{cases}
|\epsilon_t| & \text{if } t \leq n/2, \\
r\, |\epsilon_t| & \text{if } t > n/2.
\end{cases}
\end{align*}
The \emph{alternating heterogeneity} model assigns scales based on parity:
\begin{align*}
X_t = \begin{cases}
|\epsilon_t| & \text{if } t \text{ is odd}, \\
r\, |\epsilon_t| & \text{if } t \text{ is even}.
\end{cases}
\end{align*}
We use $r \in \{2, 5\}$ and truncation constant $c = 10^{-3}$.

Figure~\ref{fig:scale_hetero} displays the results. Under block heterogeneity (top row), the ABM estimator strongly outperforms both BM and SBM, and the advantage is more pronounced for larger $r$. The permutation invariance of ABM is beneficial here: by treating all observations equally regardless of their position, ABM effectively pools information from both scale regimes.

Under alternating heterogeneity (bottom row), the picture changes. For $r = 2$, the three estimators perform comparably, with BM achieving the lowest MSE for moderate $k$. For $r = 5$, BM clearly outperforms ABM. This is because with alternating scales, each disjoint block naturally contains observations from both scale regimes, resulting in block maxima that are well-behaved. In contrast, ABM's permutation invariance provides no additional benefit here, and its inclusion of lower order statistics introduces extra bias. This comparison shows that ABM's advantage under heterogeneity depends on the structure of the scale variation: ABM excels when scale changes are concentrated in blocks, but not when they alternate at fine resolution.

\begin{figure}[h!]
\centering
\subfloat[Block, $r=2$]{\includegraphics[width=0.49\textwidth]{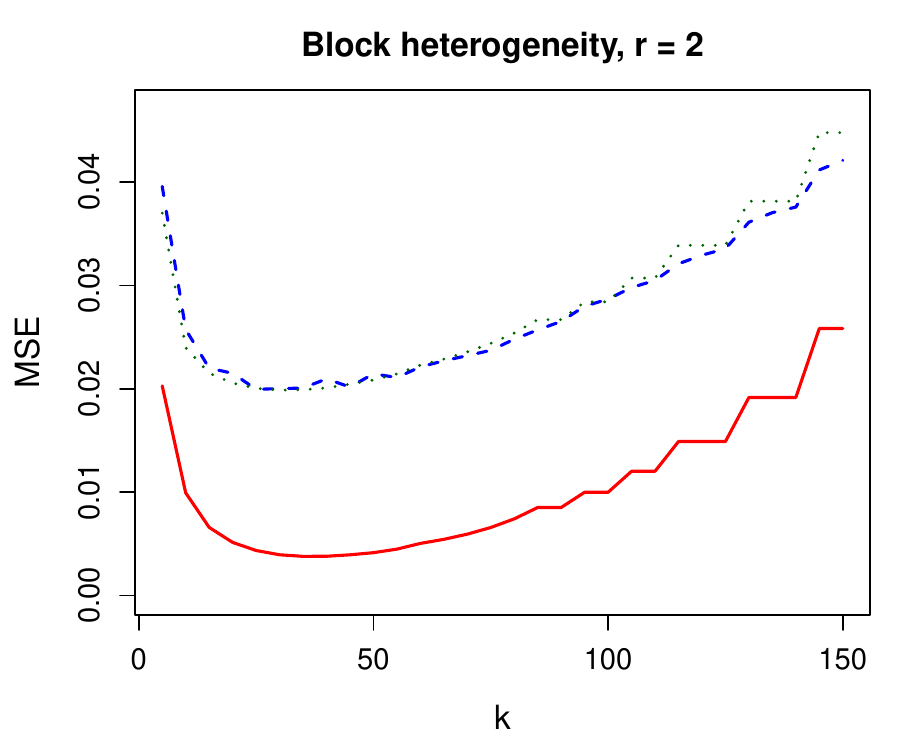}}\quad
\subfloat[Block, $r=5$]{\includegraphics[width=0.49\textwidth]{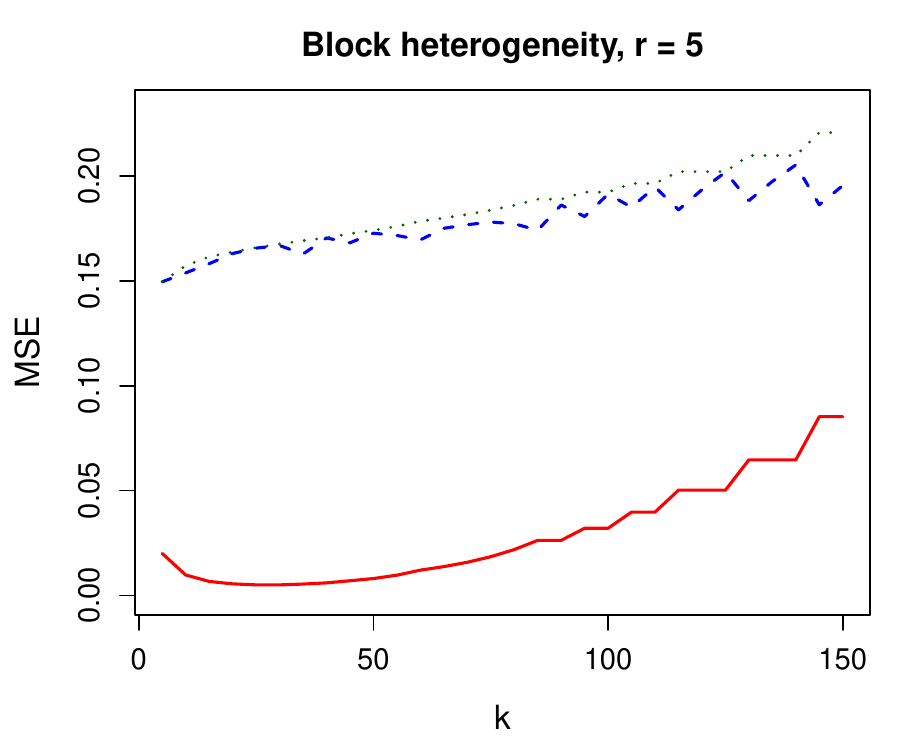}}\\
\subfloat[Alternating, $r=2$]{\includegraphics[width=0.49\textwidth]{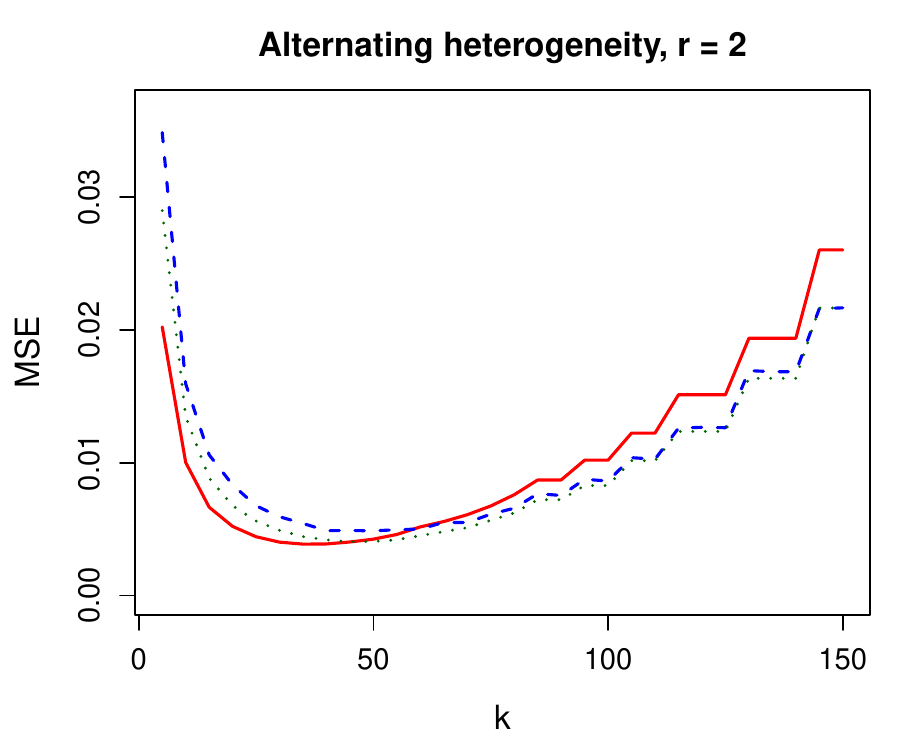}}\quad
\subfloat[Alternating, $r=5$]{\includegraphics[width=0.49\textwidth]{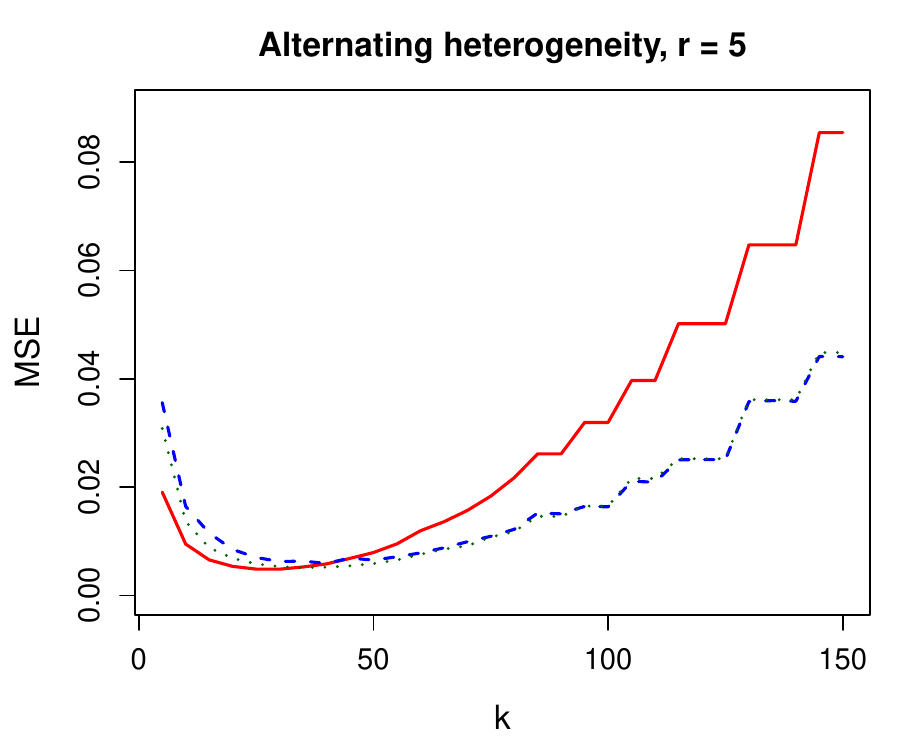}}
\caption{MSE under scale heterogeneity, $n=1000$. Top: block heterogeneity. Bottom: alternating heterogeneity. Line styles as in Figure~\ref{fig:mse_studentt}.}
\label{fig:scale_hetero}
\end{figure}

%%------------------------------------------------------------------------------------------------------------------------------------------------------
\section{Proofs}\label{Proofs}
\subsection{Preliminaries}
Recall that we use the weights $p_i=\binom{n-i}{m-1}/\binom{n}{m}$ to reweigh the truncated order statistic $X^c_{n-i+1:n}$. The following lemma shows that these weights are uniformly close to $q_i=\frac{1}{k}e^{-(i-1)/k}$.

\begin{lemma}\label{weights}
Under Condition \ref{n,k,m_condition}  we have that for any $d>0$
\begin{equation}\label{binomial_weights}
\frac{p_{i}}{\frac{1}{k}e^{-(i-1)/k}}-1=(\log k)^{4d}\suit{\frac{1}{k}+\frac{1}{m}}O(1),
\end{equation}
holds uniformly for $1\leq i\leq k (\log k)^{d}$, as $n \to \infty$
\end{lemma}
\begin{proof}
Clearly, the lemma holds for $i=1$. For $i\geq 2$, we can write the binomial weight $p_{i}$ as
$$
p_i=\frac{1}{k} \left(\frac{n-m}{n-1}\right)^{i-1} \prod_{j=1}^{i-2}\frac{1-\frac{j}{n-m}}{1-\frac{j}{n-1}},
$$
which leads to
$$\log\frac{p_{i}}{\frac{1}{k}e^{-(i-1)/k}}=-(i-1)\suit{\frac{1}{k}-\log\frac{n-m}{n-1}}+\sum_{j=1}^{i-2}\log\frac{1-\frac{j}{n-m}}{1-\frac{j}{n-1}}=:I_1+I_2.$$

For $I_1$, clearly, as $n\to\infty$, uniformly for all $i\leq k(\log k)^d$
\begin{align*}
\abs{I_1}  \leq k(\log k)^d \suit{\frac{1}{k}-\log \suit{1-\frac{1-m}{n-1}}}=\frac{(\log k)^d}{k}O(1).
\end{align*}

Next we handle $I_2$. By Condition \ref{n,k,m_condition}, for all $j<i\leq k (\log k)^{d}$ we have uniformly, $j/(n-1)\leq j/(n-m)=\frac{(\log k)^{d}}{m} O(1) \to 0$ as  $n\to\infty$. Thus
\begin{align*}
&\log\frac{1-\frac{j}{n-m}}{1-\frac{j}{n-1}}\\
=&-\left(\frac{j}{n-m}\right)-\left(\frac{j}{n-m}\right)^2+O\left(\left(\frac{j}{n-m}\right)^{3}\right)+\left(\frac{j}{n-1}\right)+\left(\frac{j}{n-1}\right)^{2}+O \left(\left(\frac{j}{n-1}\right)^{3}\right)\\
=&-\frac{j(m-1)}{(n-1)(n-m)}-\frac{j(m-1)}{(n-1)(n-m)}\suit{\frac{j}{n-m}+\frac{j}{n-1}}+\suit{\frac{j}{n}}^3O(1)\\
=&\frac{j}{nk}O(1)+\suit{\frac{j}{n}}^3O(1),
\end{align*}
where the $O(1)$ terms are uniform for all $j<i\leq k (\log k)^{d}$ as $n\to\infty$. By aggregating these terms we get that, uniform for all $i\leq k (\log k)^{d}$ as $n\to\infty$,
\begin{align*}
I_2=&\sum_{j=1}^{i-2}\log\frac{1-\frac{j}{n-m}}{1-\frac{j}{n-1}}=\sum_{j=1}^{i-2}\suit{\frac{j}{nk}O(1)+\suit{\frac{j}{n}}^3O(1)}\\
=&\frac{1}{nk}(k(\log k)^d)^2O(1)+\frac{(k(\log k)^d)^4}{n^3}O(1)\\
=&\frac{(\log k)^{2d}}{m}O(1)+\frac{k(\log k)^{4d}}{m^3}O(1)
\end{align*}
By Condition \ref{n,k,m_condition}, we get that $k=o(n^{2/3})$, which implies that $k/m^2\to 0$ as $n\to\infty$. Hence, as $n\to\infty$, $I_2=\frac{(\log k)^{4d}}{m}O(1)$. The lemma is proved by combining $I_1$ and $I_2$.
\end{proof}
In the proof of Proposition \ref{empirical_distribution}, we need to replace the deterministic term $\frac{n}{k}\left(1-F(x\sigma_m)\right)$ by its limit $x^{-1/\gamma}$. The following lemma shows that such a replacement is allowed for a wide range of $x$ uniformly. 

%Part a of the lemma shows that the ratio $\frac{\frac{n}{k}\left(1-F(x\sigma_m)\right)}{x^{-1/\gamma}} $ converges uniformly to 1 for a wide range of $x$. Part b of the lemma shows that the distance between $\frac{n}{k}\left(1-F(x\sigma_m)\right)$ and its limit $x^{-1/\gamma}$ can be uniformly bounded by a function of $k$ related to the second order scale function $a(\cdot)$ given in \eqref{SOC_condition}.
\begin{lemma}\label{diff_ratio}
Assume that conditions \ref{SOC_condition} and \ref{n,k,m_condition} hold for some $\rho< 0$. For any $\delta>0$, there exists $c_1=c_1(\delta,\rho,\gamma, \epsilon)>0$  such that as $n \to \infty$,
\begin{align}\label{replace_all}
\sup_{x> \left( \log k \right)^{-\delta}} x^{\frac{1/2-\epsilon}{\gamma}} \left| \frac{n}{k}\left(1 - F \left(\sigma_m x \right)\right) -x^{-1/\gamma} \right| = O\left(a( \sigma_m) (\log k)^{c_1}\right) \to 0
\end{align}
In addition, there exists $c_3=c_3(\delta,\rho,\gamma)>0$ such that as $n \to \infty$,
\begin{align}\label{replace_rho<0}
\sup_{x> \left( \log k \right)^{-\delta}} \left|\frac{\frac{n}{k}\left(1-F(x\sigma_m)\right)}{x^{-1/\gamma}} - 1 \right| = O\left(a(\sigma_m) (\log k)^{c_{3}} \right) \to 0.
\end{align}
\end{lemma}
\begin{proof} Convergence of the right hand sides of \eqref{replace_all}-\eqref{replace_rho<0} to zero follows by Condition \ref{n,k,m_condition}. Thus, we only show the statements regarding the order.

Recall the inequality in \eqref{uniform_ineq}. Since $\sigma_m \left( \log k \right)^{-\delta} \to \infty$ as $n\to\infty$, we get that for any $\delta, \delta', \epsilon>0$, there exists $n(\delta,\delta',\epsilon)$ such that for all $n>n(\delta,\delta',\epsilon)$ and $ x>\left( \log k \right)^{-\delta}$,
\begin{align}\label{inequality_uniform}
\left| \frac{\frac{n}{k}\left(1-F(x \sigma_m)\right)}{x^{-1/\gamma}} - 1\right| \leq \abs{ a(\sigma_m)} \left(\epsilon{ x}^{\rho/\gamma} \max\left({ x}^{\delta'}, { x}^{-\delta'}\right)+  \frac{x^{\rho/\gamma}-1}{\rho\gamma}\right).
\end{align}
%For $\rho= 0$, the right hand side of the above inequality is interpreted as the limit as $\rho \to 0$.

From \eqref{inequality_uniform}, we obtain that for all $n>n(\delta,\delta',\epsilon)$ and $x> (\log k)^{-\delta}$,
\begin{align*}
x^{\frac{1/2-\epsilon}{\gamma}} \left| \frac{n}{k}\left(1-F(x\sigma_m) \right) -x^{-1/\gamma} \right| \leq \abs{a(\sigma_m)}x^{\frac{-1/2+\epsilon}{\gamma}}\left(\epsilon{ x}^{\rho/\gamma} \max\left({ x}^{\delta'}, { x}^{-\delta'}\right)+  \frac{x^{\rho/\gamma}-1}{\rho\gamma}\right).
\end{align*}

By choosing $\delta'$ such that $\delta'<(-\rho+1/2+\epsilon)/\gamma$ we get that, for all $x>(\log k)^{-\delta}$,
$$x^{-\frac{1/2+\epsilon}{\gamma}}\left(\epsilon{ x}^{\rho/\gamma} \max\left({ x}^{\delta'}, { x}^{-\delta'}\right)+  \frac{x^{\rho/\gamma}-1}{\rho\gamma}\right)<(\log k)^{c_1},$$
with $c_1>2\delta(-\rho+1/2+\epsilon)/\gamma$, which completes the proof of \eqref{replace_all}. 

Equation \eqref{replace_rho<0} is proved in a similar way by directly handling the right hand side of \eqref{inequality_uniform}: now choose $\delta'<-\rho /\gamma$ and $c_3>-2 \delta\rho/\gamma$.

% and $c_3$: for $\rho<0$, choose $\delta'<-\rho/\gamma$ and $c_2>2\delta(-\rho)/\gamma$; for $\rho=0$, choose $\delta'=\delta$ and $c_3>\gamma\delta$.
\end{proof}

Finally, to prove Proposition \ref{empirical_distribution} we extend Theorem 5.1.4 in \cite{deHaanFerreira2006extreme} as in the lemma presented below. The extension concerns two different aspects. Firstly, in contrast to a fixed lower bound on the range of $x$, we allow the lower bound  to go to zero at a logarithmic speed. Secondly, we have an additional rate $k^{\epsilon^*}$ in the limit result.
\begin{lemma}\label{extended_process}
Under Conditions 3.1, \ref{n,k,m_condition} for any $\delta>0$ and any $0<\epsilon<1/2$ there exists $\epsilon^*=\epsilon^*(l, \rho, \epsilon)$  such that as $n \to \infty$,
\begin{align}
\sup_{x>\left( \log k \right)^{-\delta}}k^{\epsilon^*}x^{\frac{1/2-\epsilon}{\gamma}} \left|\sqrt{k} \left( \frac{n}{k} \left(1 -F_n(\sigma_m x)\right) -x^{-1/\gamma}\right) - W_{n} \left(x^{-1/\gamma}\right)\right| \overset{P}{\to} 0,
\end{align}
where $F_{n}(x)= \frac{1}{n}\sum_{i=1}^{n}\mathbb{I}\{X_{i}\leq x\}$.
\end{lemma}
\begin{proof}
We start from Corollary 3.2, Ch. $4.3$ in \cite{csorgo1993weighted}. Consider the uniform empirical distribution function, $U_{n}(t):=\frac{1}{n}\sum_{i=1}^n\mathbb{I}\{U_i\leq t\}$, where $U_1,U_2,\ldots$ are i.i.d. standard uniform random variables. Then there exists a sequence of Brownian Bridges $B_{n}$ such that for all $0<\nu <1/2$ as $n \to \infty$,
\begin{align*}
n^{1/2-\nu}  \sup_{0<t<1}\frac{\left|\sqrt{n}(U_{n}(t)-t)-B_{n}(t)\right|}{(t(1-t))^{\nu}}=O_{p}(1).
\end{align*}
By replacing $t$ with $t \frac{k}{n}$, we get that as $n \to \infty$,
\begin{align*}
k^{1/2-\nu}  \sup_{0<t<\frac{n}{k}} t^{-\nu}  \frac{\left|\sqrt{k}\left(\frac{n}{k}U_{n}\left(t\frac{k}{n}\right)-t\right)-\frac{\sqrt{n}}{\sqrt{k}}B_{n}\left(t\frac{k}{n}\right)\right|}{\left(1-t\frac{k}{n}\right)^{\nu}}=O_{p}\left(1\right).
\end{align*}
Let $z_{n}\left(x\right)= \frac{n}{k}\left(1-F\left(\sigma_mx\right)\right)$, then $z_{n}(x)<\frac{1}{2}\frac{n}{k}$ for sufficiently large $n$ when $x>( \log k )^{-\delta}$ because $\sigma_m( \log k )^{-\delta} \to \infty $. By taking $t=z_n(x)$ and $\nu=1/2-\epsilon$ we obtain that for any $\epsilon^*<\epsilon$, as $n \to \infty$,
\begin{equation} \label{eq:empirical process with zn}
k^{\epsilon^*} \sup_{x>( \log k )^{-\delta}} z_{n}\left(x\right)^{-1/2+\epsilon} \left|\sqrt{k}\left( \frac{n}{k} \left( 1-F_n(\sigma_m x) \right) -z_{n}(x)\right)  - \frac{\sqrt{n}}{\sqrt{k}} B_{n}\left(z_{n}(x)\frac{k}{n}\right)\right| \overset{P}{\to} 0.
\end{equation}
We intend to substitute the three $z_n(x)$ terms, referred to by order of appearance in  \eqref{eq:empirical process with zn} as the first, second and third $z_n(x)$ term below, by its limit $x^{-1/\gamma}$ for all $x>( \log k )^{-\delta}$ uniformly. 

The first $z_{n}\left(x\right)$ term can be substituted by its limit $x^{-1/\gamma}$ uniformly for all $x>( \log k )^{-\delta}$ by \eqref{replace_rho<0}. Hence, we have as $n \to \infty$,
\begin{align*}
k^{\epsilon^*} \sup_{x>( \log k )^{-\delta}}x^{\frac{1/2-\epsilon}{\gamma}} \left|\sqrt{k} \left( \frac{n}{k} \left(1 -F_n\left(\sigma_m x\right) \right) - z_{n}(x)\right) -  \frac{\sqrt{n}}{\sqrt{k}}B_{n} \left(z_{n}(x) \frac{k}{n}\right)  \right|\overset{P}{\to} 0.
\end{align*}
Next, to handle the second $z_{n}(x)$ term, we need to show that as $n\to\infty$,
\begin{align*}
\lim_{n\to\infty}k^{\epsilon^*} \sup_{x> \left( \log k \right)^{-\delta}} x^{ \frac{1/2-\epsilon}{\gamma}}\sqrt{k}\left|z_{n}(x)-x^{-1/\gamma}\right|=0.
\end{align*}
This relation follows directly from \eqref{replace_all} and Condition \ref{n,k,m_condition}:  as $n \to \infty$,
\begin{align*}
 k^{\epsilon^*} \sup_{x> \left( \log k \right)^{-\delta}} x^{ \frac{1/2-\epsilon}{\gamma}}\sqrt{k}\left|z_{n}(x)-x^{-1/\gamma}\right|= k^{1/2+\epsilon^*}  a(\sigma_m)(\log k)^{c_1}O(1)\to 0 .
 \end{align*}
Here the last step follows by choosing $0<\epsilon^*<-\rho-l(1/2-\rho)$. The existence of such a positive constant $\epsilon^*$ is guaranteed by Condition \ref{n,k,m_condition}, while the limit relation follows from the fact that $\abs{a(\sigma_m)}<\left(\frac{n}{k}\right)^{\rho+\epsilon^*}$ for sufficiently large $n$, see Remark \ref{remark_condition3.3}.

Lastly, we handle the third $z_n(x)$ term. Write $B_n(u)=\tilde W_n(u)-u\tilde W_n(1)$, where $\tilde W_n$ is a properly defined Brownian motion. Further define $W_n(u)=\sqrt{\frac{n}{k}}\tilde W_n\left(\frac{k}{n}u\right)$. Then $W_n(u)$ is also a Brownian motion and $\sqrt{\frac{n}{k}}B_{n} \left(z_{n}(x) \frac{k}{n}\right) =W_n\left(z_n(x) \right) -\sqrt{\frac{k}{n}}z_n(x) \tilde{W}_n(1)$. We shall prove that, as $n\to\infty$,
\begin{align}
&k^{\epsilon^*} \sup_{x> \left( \log k \right)^{-\delta}} x^{\frac{1/2-\epsilon}{\gamma}}\left|W_{n}(z_{n}(x))- W_n\left(x^{-1/\gamma}\right)\right|=o_p(1),\label{brownian_replacement}\\
&k^{\epsilon^*} \sup_{x> \left( \log k \right)^{-\delta}} x^{\frac{1/2-\epsilon}{\gamma}}\abs{z_n(x)\sqrt{\frac{k}{n}} \tilde W_n(1)}=o_p(1).
\label{brownian_replacement2}
\end{align}
We prove Equation \eqref{brownian_replacement} by splitting the region $\{ x>\left( \log k \right)^{-\delta} \}$ into two regions $\{ \left( \log k \right)^{-\delta}<x<x_0 \}$ and $\{ x\geq x_0 \}$ for a fixed constant $x_0$.

The region $\{x\geq x_0\}$ can be dealt with by the Levy's modulus of continuity for Brownian motion in the following form (see Theorem 1.1.1 and its colloraries in \cite{csorgo2014strong}):
\begin{align*}
\lim_{h \to 0} \sup_{0 \leq s \leq b }\frac{\left|W(s+h)-W(s)\right|}{\sqrt{2 h \log(1/h)}}\overset{a.s.}{=} 1,
\end{align*}
where $W(\cdot)$ is a generic Brownian Motion and any fixed $b<\infty$.

By \eqref{replace_all}, we get uniformly for all $x\geq x_0$ that $\left|z_{n}(x) -x^{-1/\gamma}\right|\to 0$ as $n\to\infty$. Consequently, for sufficiently large $n$, almost surely
\begin{align*}
\sup_{x \geq x_0} x^{\frac{1/2-\epsilon}{\gamma}}  \left| W(z_n(x))-W(x^{-1/\gamma}) \right| \leq& 2 \sup_{x \geq x_0} x^{\frac{1/2-\epsilon}{\gamma}}  \left| z_n(x) -x^{-1/\gamma}\right| ^{1/2-\epsilon}.
\end{align*}
By Condition \ref{n,k,m_condition} and \eqref{replace_rho<0}, we get that as $ n\to\infty$,
\begin{align*}
k^{\epsilon^*} \sup_{x\geq x_{0}}x^{\frac{1/2-\epsilon}{\gamma}} \left|{ W}({ z_n}({ x}))- W\left(x^{-1/\gamma}\right)\right|=k^{\epsilon^*} a( \sigma_m)^{1/2-\epsilon} (\log k)^{c_3(1/2-\epsilon)} O(1)   \to 0 \text{ a.s. } 
\end{align*}
by choosing $\epsilon^{*}<-\rho/l(1/2-\epsilon) +\rho(1/2-\epsilon)$. Therefore, to prove \eqref{brownian_replacement} it is left to handle the region $\{ \left( \log k \right)^{-\delta}<x<x_0 \}$, i.e. as $n \to \infty$,
\begin{align}\label{brownian_replacement_part2}
k^{\epsilon^*}  \sup_{(\log  k)^{-\delta} <x \leq x_0} x^{\frac{1/2-\epsilon}{\gamma}} \left|W_{n}(z_{n}(x)) - W_n\left(x^{-1/\gamma}\right) \right|\overset{P}{\to} 0.
\end{align}
Define $V_{n}(t) =tW_{n}(1/t)$. Then $V_n(t)$ is also a Brownian Motion. We bound the left hand side of Equation \eqref{brownian_replacement_part2} as follows:
\begin{align*}
&k^{\epsilon^*}x^{\frac{1/2-\epsilon}{\gamma}}\left|W_{n}(z_{n}(x))- W_n\left(x^{-1/\gamma}\right)\right| \\
=&k^{\epsilon^*}x^{\frac{1/2-\epsilon}{\gamma}} \left| \frac{z_{n}(x)}{x^{-1/\gamma}} x^{-1/\gamma} V_n\left(\frac{1}{z_{n}(x)}\right) -x^{-1/\gamma} V_n \left(x^{1/\gamma}\right) \right| \\
\leq &k^{\epsilon^*}x^{\frac{-1/2-\epsilon}{\gamma}} \left(\left|V_{n}\left(\frac{1}{z_{n}(x)}\right) \left(\frac{z_{n}(x)}{x^{-1/\gamma}}-1\right)\right|+\left|V_{n}\left(\frac{1}{z_{n}(x)}\right)-V_{n} (x^{1/\gamma})\right|\right) \\
=:&\mathbb{I}_{1}(x)+\mathbb{I}_{2}(x).
\end{align*}

To handle $\mathbb{I}_{1}(x)$, write
\begin{align*}
 \mathbb{I}_{1}(x)&=z_n(x)^{1/2-\epsilon}\abs{V_n\suit{\frac{1}{z_n(x)}}}\cdot 
 \suit{\frac{x^{-1/\gamma}}{z_n(x)}}^{1/2-\epsilon}\cdot k^{\epsilon^*}x^{-2\epsilon/\gamma}\abs{\frac{z_{n}(x)}{x^{-1/\gamma}}-1}\\
 &=:\mathbb{I}_{11}(x)\cdot \mathbb{I}_{12}(x)\cdot\mathbb{I}_{13}(x).
\end{align*}
For $(\log k)^{-\delta}<x<x_0$, \eqref{replace_rho<0} implies that for sufficiently large $n$, 
$\frac{1}{z_n(x)}<2x_0^{1/\gamma}$ uniformly and as $n\to\infty$, $\mathbb{I}_{12}(x)=O(1)$. As $n\to\infty$, since
$\sup_{s<2x_0^{1/\gamma}} s^{-1/2+\epsilon}\abs{V_n(s)}=O_p(1)$ , we conclude that $\mathbb{I}_{11}(x)=O_p(1)$ uniformly for all $(\log k)^{-\delta}<x<x_0$. Finally, for $\mathbb{I}_{13}(x)$, \eqref{replace_rho<0} implies that as $n\to\infty$, 
\begin{align*}
\sup_{(\log k)^{-\delta}<x<x_0} \mathbb{I}_{13}(x)\leq& k^{\epsilon^*}\sup_{(\log k)^{-\delta}<x<x_0} x^{-2\epsilon/\gamma}  \sup_{(\log k)^{-\delta}<x<x_0} \abs{\frac{z_{n}(x)}{x^{-1/\gamma}}-1} \\
=& k^{\epsilon^*} a(\sigma_m)(\log k)^{c_3+2\delta\epsilon/\gamma}O(1)=o(1),
\end{align*}
 where the $o(1)$ term in the last equality follows by Condition \ref{n,k,m_condition} and choosing $\epsilon^*<-\rho/l+\rho$. Combining the three components, we conclude that as $n\to\infty$,
$$\sup_{(\log k)^{-\delta}<x<x_0} \mathbb{I}_{1}(x)=o_p(1).$$

We handle $\mathbb{I}_{2}(x)$ by applying the modulus of continuity to the Brownian Motion $V_n(\cdot)$. From \eqref{replace_rho<0}, we get that as $n\to\infty$,
\begin{align*}
\sup_{\left( \log k \right)^{-\delta} <x\leq x_{0}}\left|\frac{1}{z_{n}(x)}-\frac{1}{x^{-1/\gamma}}\right| =\sup_{\left( \log k \right)^{-\delta}<x\leq x_{0}}\frac{1}{z_{n}(x)}\left|\frac{z_{n}(x)}{x^{-1/\gamma}}-1\right|=O(1)o(1)=o(1),
\end{align*}
Thus, we can apply the modulus of continuity as intended. The rest of the proof is similar to that for the region $\{ x>x_0 \}$. By combining the two components $\mathbb{I}_1(x)$ and $\mathbb{I}_2(x)$, we conclude that \eqref{brownian_replacement_part2}, and thereby \eqref{brownian_replacement}, is proved.

Finally, by choosing $\epsilon^*<1/4$, \eqref{brownian_replacement2} is proved by combining Condition \ref{n,k,m_condition} with the following facts: as $n\to\infty$, $k^{3/2}/\sqrt{n}\to 0$, $\tilde{W}_{n}(1) = O_{p}(1)$, ${ z_n}({ x})/x^{-1/\gamma} = O(1)$ and $\sup_{ x>\left( \log k \right)^{-\delta}} x ^{\frac{- 1/2-\varepsilon}{\gamma} }=O\suit{(\log(k))^{\delta(1/2+\varepsilon)/\gamma} }$.

\end{proof}
\subsection{Proof of Proposition \ref{empirical_distribution}}
\begin{proof}[Proof of Proposition \ref{empirical_distribution}]
 We shall prove the proposition in three steps: as  $n \to \infty$,
\begin{align}
&\sup_{x> \left( \log k \right)^{-\delta}} k^{\tilde{\epsilon}} x^{\frac{1/2-\epsilon}{\gamma}}\sqrt{k}\left(\mathbb{P}_{n}(x)-\mathbb{P}_{n}^{1}(x)\right)=o_{p}(1), \label{first_step} \\
&\sup_{x> \left( \log k \right)^{-\delta}} k^{\tilde{\epsilon}} x^{\frac{1/2-\epsilon}{\gamma}}\sqrt{k}\left(\mathbb{P}_{n}^{1}(x)-\mathbb{P}_{n}^{2}(x)\right)=o_{p}(1), \label{second_step} \\
&\sup_{x> \left( \log k \right)^{-\delta}}k^{\tilde{\epsilon}}x^{\frac{1/2-\epsilon}{\gamma}} \left| \sqrt{k} \left( \mathbb{P}_n^2 (x) -\left(1-e^{-x^{-1/\gamma}} \right) \right) -e^{-x^{-1/\gamma}}W_n(x^{-1/\gamma})  \right|=o_p(1), \label{third_step}
\end{align}
where
\begin{align*}
 &\mathbb{P}_{n}^{(1)}(x)=\sum_{i=1}^{n-m+1}\frac{1}{k}e^{-(i-1/k)}\mathbb{I}\{\frac{X^c_{n-i+1:n} }{\sigma_m} >x\}, \\
 &\mathbb{P}_{n}^{(2)}(x)=\sum_{i=1}^{n}\frac{1}{k}e^{-(i-1/k)}\mathbb{I}\{\frac{X^c_{n-i+1:n} }{\sigma_m} >x\}.
\end{align*}
Compared to $\mathbb{P}_{n}$, $\mathbb{P}_{n}^{(1)}$ uses the weights $q_i$ instead of $p_i$ and $\mathbb{P}_{n}^{(2)}$ further includes the lowest $m-1$ order statistics.

Denote $\bar{F}_{n}(x) =  \frac{1}{n}\sum_{i=1}^{n}\mathbb{I}\{X_{i} > x\}$. Note that for sufficiently large $n$ we have $\sigma_m x>c$ for all $x>\left(  \log k \right) ^{-\delta}$, which implies $\{X^c_{n-i+1:n} >\sigma_m x\}=\{ \bar{F}_n(\sigma_m x) \geq  \frac{i}{n}\}$ holds for all $1\leq i \leq n-m+1$ and  $x>\left(  \log k \right) ^{-\delta}$. Consequently, we get that for sufficiently large $n$,
\begin{align*}
&k^{\tilde{\epsilon}}x^{\frac{1/2-\epsilon}{\gamma}}\sqrt{k}\ \left| \mathbb{P}_{n}(x)- \mathbb{P}_{n}^{(1)}(x) \right| \\
\leq &k^{\tilde{\epsilon}}x^{\frac{1/2-\epsilon}{\gamma}}\sqrt{k}\sum_{i=1}^{n-m+1} \left| p_{i} - \frac{1}{k}e^{-(i-1)/k}\right|\ \mathbb{I}\{X^c_{n-i+1:n} >\sigma_m x\} \\
=&k^{\tilde{\epsilon}}x^{\frac{1/2-\epsilon}{\gamma}}\sqrt{k} \sum^{ n\bar{F}_n(x \sigma_m)}_{ i=1} \frac{1}{k}e^{-(i-1)/k} \left|\frac{p_{i}}{q_{i}}- 1\right|.
\end{align*}

 We intend to apply Lemma \ref{weights} to bound the terms $\left|\frac{p_{i}}{q_{i}}- 1\right|$ for $i \leq n \bar{F}_n(x \sigma_m) $. Notice that Lemma \ref{extended_process} implies that, for any $\eta>0$ and sufficiently large $n$, on a set $A_n$ with $\mathbb{P}(A_n)>1-\eta$ we have 
\begin{align*}
n \bar{F}_n(\sigma_m (\log k)^{-\delta}) \leq& k^{1/2-\epsilon^*}( \log k )^{\delta \left(\frac{1/2-\epsilon}{\gamma} \right)} +\sqrt{k}\left| W_n\left((\log k)^{\frac{\delta}{\gamma}}\right) \right| +k (\log k)^{\delta / \gamma}.\\
\leq & k (\log k)^d,
\end{align*}
where $d> \delta /\gamma$.
Hence by Lemma \ref{weights} we obtain that for sufficiently large $n$, on the set $A_n$, there exists a constant $C$ such that
\begin{align*}
&k^{\tilde{\epsilon}}x^{ \frac{1/2-\epsilon}{\gamma}}\sqrt{k}\left| \mathbb{P}_{n}(x)-\mathbb{P}_{n}^{1}(x)\right|\\
\leq&C\left( (\log k)^{4d}\left(\frac{1}{k}+\frac{1}{m}\right)\right)k^{\tilde{\epsilon}}x^{\frac{1/2-\epsilon}{\gamma}}\sqrt{k} \sum^{ n\bar{F}_n(x \sigma_m)}_{ i=1} \frac{1}{k}e^{-(i-1)/k}\\
=&C\left( (\log k)^{4d}\left(\frac{1}{k}+\frac{1}{m}\right)\right)k^{\tilde{\epsilon}}x^{\frac{1/2-\epsilon}{\gamma}}\sqrt{k} \left( 1-e^{-\frac{n}{k}\bar{F}_n\left(\sigma_m x \right)  } \right),
\end{align*}
uniformly for $x > \left( \log k \right)^{-\delta}$. By the taylor expansion of the function $\phi(y)=1-e^{-y}$ at the point $x^{-1/\gamma}$, we obtain that, there exists $\tilde{x}_{n}$ between $x^{-1/\gamma}$ and $\frac{n}{k} \bar{F}_n(\sigma_m x)$ such that
\begin{align*}
 x^{\frac{1/2-\epsilon}{\gamma}} \left(1-e^{-\frac{n}{k} \bar{F}_n(\sigma_m x) }\right) =  x^{\frac{1/2-\epsilon}{\gamma}}(1 -e^{-x^{-1/\gamma}}) + x^{\frac{1/2-\epsilon}{\gamma}}e^{-\tilde{x}_{n}} \left( \frac{n}{k}\bar{F}_{n} (\sigma_m x) -x^{-1/\gamma}\right).
\end{align*}
Note that  $\sup_{x>0}x^{\frac{1/2-\epsilon}{\gamma}}\left(1 -e^{-x^{-1/\gamma}}\right) $ is finite. Together with the fact that $0<e^{-\tilde{x_n}}<1$ and Lemma \ref{extended_process}, we conclude that $\sup_{x> (\log k)^{-\delta}} x^{\frac{1/2-\epsilon}{\gamma}} \left(1-e^{-\frac{n}{k} \bar{F}_n(\sigma_m x) }\right)=O_p(1)$. By Condition \ref{n,k,m_condition}, $ k = O(n^{l})$ with $l<2/3$. Therefore we can take $\tilde{\epsilon}<1/l-3/2$ small enough and have that $k^{1/2+\tilde{\epsilon}} (\log k)^{4d} /m\to 0$. With this $\tilde{\epsilon}$ we conclude that \eqref{first_step} is proved.  

To prove \eqref{second_step}, note that $\{\left| \mathbb{P}_{n}^{1}(x)-\mathbb{P}_{n}^{2}(x)\right| \neq 0 \} \subset \{ X^c_{m:n}>\sigma_m x\} $. Moreover, for any constant $0 < a< 1$ we have eventually $X_{m:n} \leq X_{na:n}$ while $X_{na:n}\overset{P}{\to} F^{-1}(a) <\infty$ and $\left( \log k \right)^{-\delta}\sigma_m \to\infty$, as $n \to \infty$. Therefore $\mathrm{P}\left( X_{m:n}^c>\sigma_m x \right)\to 0$ as $n \to \infty$, which proves \eqref{second_step}.

To prove \eqref{third_step}, recall the equivalence between the events $\{X^c_{n-i+1:n} >\sigma_m x\}$ and $\{ \bar{F}_n(\sigma_m x) \geq  \frac{i}{n}\}$ for large enough $n$ and $\phi(y)=1-e^{-y}$. Write
\begin{align*}
 \mathbb{P}_{n}^{(2)}(x)=&\sum_{i=1}^{n}\frac{1}{k}e^{-(i-1)/k}\mathbb{I}\{\frac{X^c_{n-i+1:n} }{\sigma_m} >x\} \\
 =&\sum_{i=1}^{n}\frac{1}{k} { e}^{-({ i}-1)/k}\mathbb{I}\{\bar{F}_n(\sigma_m x) \geq  \frac{i}{n}\} \\
=&\sum^{n\bar{F}_n(\sigma_m x)}_{i=1}  \frac{1}{k}e^{-(i-1)/k} = \frac{1-e^{-\frac{n}{k}\bar{F}_{n}(\sigma_{m}x)}}{k(1-e^{-1/k})}=\left(1+\frac{1}{k}o(1) \right) \phi\left( \frac{n}{k}\bar{F}_{n}(\sigma_{m}x) \right).
\end{align*}
Since
\begin{align*}
\phi\left(\frac{n}{k} \bar{F}_n\left(\sigma_m x\right)\right) -\phi\left(x^{-1/\gamma}\right) =e^{-x^{-1/\gamma}} \left( \frac{n}{k}\bar{F}_n(\sigma_m x) - x^{-1/\gamma}\right) -  \frac{e^{-\tilde{x}_{n}}}{2} \left( \frac{n}{k}\bar{F}_{n} (\sigma_m x) -x^{-1/\gamma}\right)^2,
\end{align*}
where $\tilde{x}_{n}$ is in between $x^{-1/\gamma}$ and  $\frac{n}{k} \bar{F}_n(\sigma_m x)$, by Lemma \ref{extended_process} we obtain that as $n \to \infty$,
\begin{align*}
\sup_{x\geq \left( \log k \right)^{-\delta}} x^{\frac{1/2-\epsilon}{\gamma}} \sqrt{k}\left( \frac{n}{k} \bar{F}_n(\sigma_m x) -x^{-1/\gamma}\right) =O_{p}(1) .
\end{align*}
By choosing  $\tilde{\epsilon}<1/2$ we get that as $n \to \infty$,
\begin{align*}
{ k}^{\tilde{\epsilon}}\sup_{x\geq \left( \log k \right)^{-\delta}} x^{\frac{1/2-\epsilon}{\gamma}} \sqrt{k} \left( \frac{n}{k}\bar{F}_{n} (\sigma_m x) -x^{-1/\gamma}\right)^2 = o_{p}(1).
\end{align*}
Combining the aforementioned limit relations with the fact that $0<e^{-\tilde{x}_n}<1$, we get that as $n \to \infty$, uniformly for all $x> (\log k)^{-\delta}$, 
\begin{align}
\phi\left( \frac{n}{k}\bar{F}_{n} (\sigma_m x)\right) = \phi\left(x^{-1/\gamma}\right)+&e^{-x^{-1/\gamma}}\left(\frac{n}{k}\bar{F}_{n} \left( \sigma_m x)\right)-x^{-1/\gamma}\right)+k^{-1/2-\tilde{\epsilon}}x^{\frac{-1/2+\epsilon}{\gamma}}o_p(1) \nonumber \\
= \phi\left(x^{-1/\gamma}\right)+&e^{-x^{-1/\gamma}} \left(k^{-1/2} W_n\left(x^{-1/\gamma}\right)+k^{-1/2-\epsilon^{*}}x^{\frac{-1/2+\epsilon}{\gamma}}o_p(1)    \right) \nonumber \\
+&k^{-1/2-\tilde{\epsilon}}x^{\frac{-1/2+\epsilon}{\gamma}}o_p(1) \nonumber \\
= \phi\left(x^{-1/\gamma}\right)+&k^{-1/2}e^{-x^{-1/\gamma}}W_n\left(x^{-1/\gamma}\right)+k^{-1/2-\tilde{\epsilon}}x^{\frac{-1/2+\epsilon}{\gamma}}o_p(1), \label{phi_expansion1}
\end{align}
where in the last step we choose $\tilde{\epsilon}<\epsilon^*$.
 
Since $\sup_{x> (\log k)^{-\delta}} x^{\frac{1/2-\epsilon}{\gamma}}e^{-x^{-1/\gamma}}W_n(x^{-1/\gamma})=O_p(1) $, as $n \to \infty$, and $\sup_{x >( \log k)^{-\delta} } x^{\frac{1/2-\epsilon}{\gamma}} \phi(x^{-1/\gamma})$ is finite, the relation \eqref{phi_expansion1} further implies that
\begin{align}
\phi\left( \frac{n}{k}\bar{F}_{n} (\sigma_m x)\right) =x^{\frac{-1/2+\epsilon}{\gamma}}O(1) +k^{-1/2}x^{\frac{-1/2+\epsilon}{\gamma}}O_p(1)=x^{\frac{-1/2+\epsilon}{\gamma}}O_p(1),  \label{phi_expansion2}
\end{align}
uniformly for $x> (\log k)^{-\delta}$. Therefore, as $n \to \infty$ uniformly for $x> \left( \log  k  \right)^{-\delta}$, 
\begin{align*}
\mathbb{P}_{n}^{(2)}(x) &= \phi\left(\frac{n}{k}\bar{F}_n(\sigma_m x) \right)\left(1+\frac{1}{k}o(1)\right)\\
&=\phi\left(\frac{n}{k}\bar{F}_n(\sigma_m x) \right)+\frac{1}{k}\phi\left(\frac{n}{k}\bar{F}_n(\sigma_m x) \right)o(1)\\
&=\phi\left(x^{-1/\gamma} \right)+k^{-1/2}e^{-x^{-1/\gamma}}W_n\left(x^{-1/\gamma}\right)+k^{-1/2-\tilde{\epsilon}}x^{\frac{-1/2+\epsilon}{\gamma}}o_p(1)+\frac{1}{k}x^{\frac{-1/2+\epsilon}{\gamma}}O_p(1).
\end{align*}
Here in the last step we replace $\phi\left(\frac{n}{k}\bar{F}_n(\sigma_m x) \right)$ in the two consecutive terms by the expressions in \eqref{phi_expansion1} and \eqref{phi_expansion2} respectively. Therefore we have completed the proof of \eqref{third_step}.
\end{proof}

\subsection{Proof of Theorem \ref{main_theorem}}
\begin{proof}[Proof of Theorem \ref{main_theorem}] The theorem follows by applying Theorem 2.5 in \cite{bucher2018maximum} to the truncated sample of all block maxima
$\underbrace{X^c_{n:n}}_{\binom{n-1}{m-1}}, \cdots, \underbrace{X^c_{n-i+1:n} }_{\binom{n-i}{m-1}}, \cdots,  \underbrace{X^c_{m:n}}_{1}$. To apply the theorem we need to verify its two sets of conditions.\\
\\
The first set of conditions requires that there exists $\gamma_-$ and $\gamma_+$ satisfying $0 <\gamma_- <\gamma_0 <\gamma_+ < \infty$ such that for all $f \in \mathbb{F}(\gamma_-, \gamma_+) := \{x \to \log(x)\}\cup\{x\to x^{-1/\gamma} : \gamma\in (\gamma_-, \gamma_+)\}\cup \{x\ \to\ x^{-1/\gamma} \log(x)\ :\ \gamma\in\ (\gamma_-,\ \gamma_+)\}\cup\{x \to x^{-1/\gamma} \log^{2}(x) : \gamma\in(\gamma_-, \gamma_+)\}$
we have as $n \to \infty$
\begin{align*}
 \sum_{i=1}^{n-m+1}p_{i}f \left(\frac{X^c_{n-i+1:n} }{\sigma_m}\right)\overset{P}{\to}  \int_{0}^{\infty}f(x) d \left( e^{-x^{-1/\gamma_0}} \right),
\end{align*}
For the second set of conditions we will show that, as $n \to \infty$,
\begin{align*}
 \sqrt{k} & \left(   \sum_{i=1}^{n-m+1}p_{i}f_{1}\left(\frac{X^c_{n-i+1:n} }{\sigma_m}\right)- \int_{0}^{\infty} f_{1}(x) d \left( e^{-x^{-1/\gamma_0}} \right), \right. \\
& \left. \sum_{i=1}^{n-m+1}p_{i}f_{2}\left(\frac{X^c_{n-i+1:n} }{\sigma_m}\right) - \int_{0}^{\infty} f_{2}(x)d \left( e^{-x^{-1/\gamma_0}} \right), \right. \\
 & \left. \sum_{i=1}^{n-m+1}p_{i}f_{3}\left(\frac{X^c_{n-i+1:n} }{\sigma_m}\right) - \int_{0}^{\infty} f_{3}(x) d \left( e^{-x^{-1/\gamma_0}} \right) \right) \\
 &\overset{P}{\to} \mathbf{Y}=: \left(Y_{1},\ Y_{2},\ Y_{3}\right)^{T} ,
\end{align*}
for $f_1=x^{-1/\gamma_0}\log(x)$, $f_2=x^{-1/\gamma_0}$ and $f_3=\log(x)$ with
\begin{align*}
&Y_{1}= \int_{0}^{\infty} f'_1(x) e^{-x^{-1/\gamma_0}} { W_n}(x^{-1/\gamma_0} ) dx,\\
&Y_{2}= \int_{0}^{\infty}f'_2(x) e^{-x^{-1/\gamma_0}}  W_n\left(x^{-1/\gamma_0}\right){ dx},\\
&Y_3= \int_{0}^{\infty}f'_3(x) e^{-x^{-1/\gamma_0}} { W_n}(x^{-1/\gamma_0} ) dx.
\end{align*}
The two sets of conditions are checked by applying Proposition \ref{empirical_distribution}, where checking the first set is similar to checking the second set, but simpler. Therefore we only show that the second set of conditions holds. 

Starting with $f_{1}(x)= x^{-1/\gamma_0} \log(x)$, we write $\gamma$ in replacement of $\gamma_0$ and show that the following relations holds: for some $\delta>0$, as $n \to \infty$,
\begin{align}
&\sqrt{k}\sum_{i=1}^{k (\log k )^2}p_{i}  f_1\left( \frac{X^c_{n-i+1:n} }{\sigma_m}\right) \mathbb{I}\{\frac{X^c_{n-i+1:n} }{\sigma_m} \leq\left( \log k \right)^{-\delta}\} =o_{p}(1), \label{first_condition} \\
&\sqrt{k}\sum_{i=k (\log k )^2}^{n-m+1}p_{i} f_1\left(\frac{X^c_{n-i+1:n} }{\sigma_m}\right)\mathbb{I}\{\frac{X^c_{n-i+1:n} }{\sigma_m} \leq\left( \log k \right)^{-\delta}\} =o_{p}(1), \label{second_condition}  \\
&\sqrt{k} \left( \sum^{n-m+1}_{i=1} p_{i} f_1\left(\frac{X^c_{n-i+1:n} }{\sigma_m}\right) \mathbb{I}\{ \frac{X^c_{n-i+1:n} }{\sigma_m} > \left( \log k \right)^{-\delta}\}  
 -\int_{0}^{\infty}\ f_1(x) d \left( e^{-x^{-1/\gamma}} \right) \right)   \nonumber    \\
=&\int_{0}^{\infty} f_1'(x) e^{-x^{-1/\gamma}} W_n\left(x^{-1/\gamma}\right){ dx} + o_{p}(1). \label{third_condition}
\end{align}
Note that relation \eqref{first_condition} holds by choosing $\delta>2\gamma$: for $i \leq k \left(\log k\right) ^2$,  $\mathbb{I}\{\frac{X_{n-i+1}\vee c}{\sigma_m} \leq\left( \log k \right)^{-\delta}\} =o_{p}(1)$ uniformly as $n \to \infty$ by Condition \ref{n,k,m_condition}.

 To show relation \eqref{second_condition}, we write for any $\epsilon>0$, as $n \to \infty$,
\begin{align*}
& \sqrt{k}\sum_{i=k \left(\log k\right) ^2 }^{n-m+1}p_{i}f_1\left(\frac{X^c_{n-i+1:n} }{\sigma_m}\right)\mathbb{I}\{\frac{X^c_{n-i+1:n} }{\sigma_m} \leq\left( \log k \right)^{-\delta}\} \\
=&O \left(1 \right) \sum_{i=k \left(\log k\right) ^2}^{n-m+1}p_{i} \left(\frac{c}{\sigma_m}\right) ^{-1/\gamma} \log\left( \frac{c}{\sigma_m}  \right)\\
=&O \left(1 \right) m^{1+\epsilon} \sqrt{k} \log(m)\sum_{i=k \left(\log k\right) ^2}^{n-m+1}p_{i}\\
=& O \left(1 \right)m^{1+\epsilon} \sqrt{k} \log(m) \binom{n-k \left(\log k\right) ^2+1}{m}/\binom{n}{m}\\
=&O \left(1 \right) m^{1+\epsilon}\sqrt{k} \log(m) \exp\left(\sum_{i=1}^{m-2}\log\left(1-\frac{\left(\log k\right) ^2}{m}+\frac{1}{n}-\frac{i}{n}\right)-\log\left(1-\frac{1}{n}-\frac{i}{n}\right)\right)\\
=&O \left(1 \right) m^{1+\epsilon} \sqrt{k} \log(m)  k^{- \log k} =o(1),
\end{align*}
where in the second equality we use that $\sigma_m^{1/\gamma}<m^{1+\epsilon}$ (see Remark \ref{remark_condition3.3}), in the second last equality we use the taylor expansion of $\log(1-x)$ and in the last step we use the condition that $k>n^{l'}$ for some $l'>0$ as $n$ is large enough, see Condition \ref{n,k,m_condition}.

To prove relation \eqref{third_condition} we write
\begin{align*}
 &\sum^{n-m+1}_{i=1} p_{i} f_1\left(\frac{X^c_{n-i+1:n} }{\sigma_m}\right)\mathbb{I}\{ \frac{X^c_{n-i+1:n} }{\sigma_m} > \left( \log k \right)^{-\delta}\} \\
=& \sum^{n-m+1}_{i=1} p_{i} \left( \int_{\left( \log k \right)^{-\delta}}^{\infty} f_1'(x) \mathbb{I}\{ \frac{X^c_{n-i+1:n} }{\sigma_m} > x\} dx \right. \\
& \left. \hspace{20mm}+ \log(\left( \log k \right)^{-\delta})\left( \log{ k}\right)^{ \delta /\gamma} \mathbb{I}\{\frac{X^c_{n-i+1:n} }{\sigma_m} >\left( \log k \right)^{-\delta}\} \right) \\
=:& \mathbb{I}_{1}+\mathbb{I}_{2}.
\end{align*}
By applying Proposition \ref{empirical_distribution} with $\delta > 2\gamma$ we obtain for ${\mathbb{I}}_{1}$ that, as $n \to \infty$,
\begin{align*}
 \sqrt{k}\left({\mathbb{I}}_{1}-\int_{\left( \log k \right)^{-\delta}}^{\infty} f_1'(x) \left(1-e^{-x^{-1/\gamma}}\right){ dx}\right) 
 =\int_{(\log k)^{-\delta}}^{\infty} f_1'(x) e^{-x^{-1/\gamma}}  W_n\left(x^{-1/\gamma}\right){ dx} +o_{p}(1).
\end{align*}
Note that here the $o_{p}(1)$ term arises with the aid of the rate $k^{\tilde{\epsilon}}$ and the weight function $x^{\frac{1/2-\epsilon}{\gamma}}$ in Proposition \ref{empirical_distribution}.

For ${\mathbb{I}}_{2}$, by applying Proposition \ref{empirical_distribution} evaluated at $x=\left( \log k \right)^{-\delta}$, we obtain
\begin{align*}
\sqrt{k} \left({\mathbb{I}}_{2}-\left(1-\exp\left(\left( \log k \right)^{\delta/\gamma}\right)\right)\log\left(\left( \log k \right)^{-\delta}\right) \log\left({ k}\right)^{ \delta/\gamma} \right) =o_{p}(1).
\end{align*}
To combine the results for ${\mathbb{I}}_{1}$ and ${\mathbb{I}}_{2}$ presented above, we have by partial integration that
\begin{align*}
& \int_{\left( \log k \right)^{-\delta}}^{\infty} f_1'(x)(1-e^{-x^{-1/\gamma}}){ dx} \\
&=-\left(1-\exp\left(\left( \log k \right)^{\delta/\gamma}\right)\right)\log\left(\left( \log k \right)^{-\delta}\right) \log\left({ k}\right)^{ \delta/\gamma}+ 
 \int_{\left( \log k \right)^{-\delta}}^{\infty}f_1(x) d \left(e^{-x^{-1/\gamma}}\right).
\end{align*}
Therefore, as $n \to \infty$,
\begin{align*}
\sqrt{k} \left( {\mathbb{I}}_{1}+{\mathbb{I}}_{2}-\int_{\left( \log k \right)^{-\delta}}^{\infty} f_1(x) d \left(e^{-x^{-1/\gamma}}\right) \right) =
 \int_{ (\log k )^{-\delta}}^{\infty} f_1'(x) e^{-x^{-1/\gamma}} W_n\left(x^{-1/\gamma}\right){ dx} +o_{p}(1).
\end{align*}
Note that one can extend the domain of $W_n(t)$ to $0<t<\infty$ by the independent increment property of Brownian Motion. Therefore, to get the result in \eqref{third_condition} it is only left to prove that, as $n \to \infty$,
\begin{align}
 &\sqrt{k}\int_{0}^{\left( \log k \right)^{-\delta}}f_1(x) d \left(e^{-x^{-1/\gamma}}\right) =o(1), \label{result_here} \\
&\int_{ 0}^{(\log k )^{-\delta}} f_1'(x) e^{-x^{-1/\gamma}} W_n\left(x^{-1/\gamma}\right){ dx}=o_p(1). \label{result_appendix}
\end{align}
For \eqref{result_here}, note that for $\delta>2\gamma$, we have that as $n \to \infty$,
\begin{align*}
&\left| \sqrt{k}\int_{0}^{\left( \log k \right)^{-\delta}} f_1(x) d \left(e^{-x^{-1/\gamma}}\right) \right| =\left| - \gamma \sqrt{k}\int_{\left( \log k \right)^{\frac{\delta}{\gamma}}}^{\infty}u \log(u) e^{-u} d u \right| \\
\leq &  \gamma \sqrt{k}\int_{\left( \log k \right)^{2}}^{\infty} e^{-u/2} d u=2 \gamma \sqrt{k} e^{-(\log k)^2 /2} \to 0.
\end{align*}
For \eqref{result_appendix}, omitting the subscript $n$ and regarding $W(.)$ as a generic Brownian Motion, it suffices to prove that as $n \to \infty$,
\begin{align*}
\lim_{n \to \infty} \int_{ 0}^{(\log k )^{-\delta}} f_1'(x) e^{-x^{-1/\gamma}} W\left(x^{-1/\gamma}\right){ dx}=0 \text{ a.s}.
\end{align*}
Note that by the law of the iterated logarithm we have that, as $x \to 0$,
\begin{align*}
\lim_{x \to 0} \left( x^{-1/\gamma} \right) ^{-1/2-\epsilon} W\left(x^{-1/\gamma}\right) =0 \text{ a.s.}
\end{align*}
Furthermore, note that $f_1'(x)=-\frac{1}{\gamma} x^{-1/\gamma-1} \log x + x^{-1/\gamma-1}$, which implies that
\begin{align*}
\lim_{x \to 0} e^{-x^{-1/\gamma}}f_1'(x) \left( x^{-1/\gamma} \right) ^{1/2+\epsilon}=0.
\end{align*}
Therefore we have that, as $n\to \infty$,
\begin{align*}
&\left| \int_{ 0}^{(\log k )^{-\delta}} f_1'(x) e^{-x^{-1/\gamma}} W\left(x^{-1/\gamma}\right){ dx} \right| \leq  \int_{ 0}^{(\log k )^{-\delta}}\left|e^{-x^{-1/\gamma}}f_1'(x) \left( x^{-1/\gamma} \right) ^{1/2+\epsilon} \right|{ dx} \\
&\leq (\log k )^{-\delta} \to 0 \text{ a.s. },
\end{align*}
which completes the proof of \eqref{result_appendix}. The steps to prove the results for $f_2$ and $f_3$ are similar and the calculation of the covariance matrix of $\mathbf{Y}$ is given in Appendix A.
\end{proof}

%%-------------------------------------------------------------------------------------------------------------------------------------------------

%%%%%%%%%%%%%%%%%%%%%%%%%%%%%%%%%%%%%%%%%%%%%%
%% Single Appendix:                         %%
%%%%%%%%%%%%%%%%%%%%%%%%%%%%%%%%%%%%%%%%%%%%%%
\begin{appendix}
\section{Appendix A}\label{appendix_A}
By the proof of Theorem \ref{main_theorem} we obtain that as $n \to \infty$,
\begin{align*}
\left( \frac{1}{\hat{\gamma}_n} -\frac{1}{\gamma},\hspace{1mm} \hat{\sigma}_n/\sigma -1  \right)^{T} \overset{d}{\to} M \left(Y_{1},\ Y_{2},\ Y_{3}\right)^{T},
\end{align*}
where the matrix $M$ is given in Theorem \ref{main_theorem} and recall that 
\begin{align*}
Y_{1}=& \int_{0}^{\infty} {x}^{-1/\gamma-1}\left(1- \frac{\log(x)}{\gamma}\right) e^{-x^{-1/\gamma}} { W_n}(x^{-1/\gamma} ) dx=\gamma \int_{0}^{\infty}( 1+\log u ) e^{-u} { W_n}(u ) du, \\
Y_{2}=& \int_{0}^{\infty}\frac{x^{-1/\gamma-1}}{-\gamma} e^{-x^{-1/\gamma}}  W_n\left(x^{-1/\gamma}\right){ dx}=- \int_{0}^{\infty} e^{-u}  W_n(u)du,\\
Y_3=& \int_{0}^{\infty}\frac{1}{x} e^{-x^{-1/\gamma}} { W_n}(x^{-1/\gamma} ) dx =\gamma \int_{0}^{\infty}u^{-1} e^{-u} { W_n}(u ) du.
\end{align*}
Note that $\mathbf{Y}= \left(Y_{1},\ Y_{2},\ Y_{3}\right)^{T}$ is a Gaussian random vector with mean zero. We characterize its distribution by calculating the covariance matrix. Define $Ei(x)= \int_{-x}^{\infty}\frac{e^{-t}}{t}dt$ and denote $\tau$ as the Euler-Mascheroni constant. As an example, we show the explicit calculation for the variance of $Y_1$.\footnote{The other explicit calculations are available on request.}
\begin{align*}
\text{Var}(Y_1)=&\mathbb{E}\gamma^2  \int_{0}^{\infty}\int_{0}^{\infty}e^{-u}(1+\log(u))e^{-x}(1+\log(x))W_{n}(x)W_{n}(u)dudx\\
=&\gamma^2  \int_{0}^{\infty}\left[ \int_{0}^{x}u(1+\log(u))(1+\log(x))e^{-u}e^{-x}du \right.\\
&+\left. \int_{x}^{\infty}e^{-u}e^{-x}x(1+\log(x))(1+\log(u))du)\right]dx\\
%=&\gamma^2  \int_{0}^{\infty}e^{-x}(1+\log(x))[Ei(-u)-ue^{-u}-e^{-u}(u+1)\log(u)-2e^{-u}]_{0}^{x}dx\\
%&+\gamma^2  \int_{0}^{\infty}xe^{-x}(1+\log(x))[Ei(-u)-e^{-u}\log(u)-e^{-u}]_{x}^{\infty}dx\\
=&\gamma^2  \int_{0}^{\infty}e^{-x}(1+\log(x))(Ei(-x)-xe^{-x}-e^{-x}(x+1)\log(x)-2e^{-x}-\tau+2) { dx}\\
&+\gamma^2  \int_{0}^{\infty}xe^{-x}(1+\log(x))(-Ei(-x)+e^{-x}\log(x)+e^{-x})dx\\
=&\gamma^2 \big( \int_{0}^{\infty}(1-x)e^{-x}(1+\log(x))Ei(-x)dx-\int_{0}^{\infty}e^{-2x}(1+\log(x))\log(x)dx  \\
&-2  \int_{0}^{\infty}e^{-2x}(1+\log(x))dx+(2-\tau)\int_{0}^{\infty}e^{-x}(1+\log(x)dx)\big)\\
=&\gamma^2 (1/2 (\tau+\log (8)-1)- \frac{1}{12} (\pi^2-6 (\tau+\log(2))+6(\tau+\log (2))^{2})  \\
&-(1-\tau-\log(2))+(2-\tau)(1 -\tau))=: \gamma^2 p \approx 0.706 \gamma^2.
\end{align*}
The final calculated covariance matrix of $\mathbf{Y}$ is given by
$$
\text{Cov}(\mathbf{Y})=
\begin{bmatrix}
\gamma^2 p &-&-\\
-\frac{\gamma}{2}(1-\tau +\log(2))&\frac{1}{2}&- \\
\gamma^2 ((3-\tau-\frac{\log(2)}{2})\log(2) -\frac{\pi^2}{12}  )&-\gamma \log(2)&\gamma^2 2 \log(2)
\end{bmatrix}.
$$

We check the calculation for $\text{Cov}(\mathbf{Y})$ numerically. Since $\mathbb{E}\left(W_n(u)W_n(s)\right)=u \wedge s$, note that we can express the covariances with the use of independent standard exponential random variables $U$ and $S$, having joint density $f_{U,S}(u,s)=e^{-(u+s)}$, as follows:
\begin{align*}
\text{Cov}(Y_i,Y_j)=\mathbb{E}\left(g_i(U) g_j(S) \left(U \wedge S\right) \right), \text{ for } i,j=1,2,3,
\end{align*}
where $g_1(u)=\gamma(1+\log u)$, $g_2(u)=-1$ and $g_3(u)=\gamma u^{-1}$. We simulate $100,000$ copies of independent standard exponential random variables $(U_m, S_m)$ each and obtain an estimate for the covariances as
\begin{align*}
 \widehat{\text{Cov}(Y_i,Y_j)}=\frac{1}{100000} \sum_{m=1}^{100000}g_i(U_m) g_j(S_m) \left(U_m \wedge S_m\right) \text{ for } i,j=1,2,3.
\end{align*}
The resulting estimate for the covariance matrix, rounded to two decimal places, is presented below:
\begin{align*}
\widehat{\text{Cov}(\mathbf{Y})}=\begin{bmatrix}
0.7 \gamma^2  &-&-\\
-0.56\gamma&0.5&- \\
0.62 \gamma^2 &-0.69\gamma &1.38\gamma^2 
\end{bmatrix}.
\end{align*}
It agrees with the calculated covariance matrix.
%(\frac{1}{2}(\tau +\log(8)-1) -\frac{1}{12}(\pi^2-6(\tau+\log(2))+6(\tau+\log(2))^2) -(1-\tau-\log(2))+(2-\tau)(1-\tau)$.

\end{appendix}
\bibliographystyle{plainnat} % Style BST file (imsart-number.bst or imsart-nameyear.bst)
\bibliography{abm}       % Bibliography file (usually '*.bib')

%% or include bibliography directly:

\end{document}